\documentclass[12pt]{article}
\usepackage{amsmath,amssymb,amsthm,amscd}
\usepackage[all]{xy}
\usepackage{graphicx}
\usepackage{tikz}
\usepackage{color}
\newenvironment{dedication}
        {\vspace{1cm}\begin{quotation}\begin{center}\begin{em}}
        {\par\end{em}\end{center}\end{quotation}}

\newcommand{\Z}{\mathbb{Z}}
\newcommand{\Q}{\mathbb{Q}}
\newcommand{\R}{\mathbb{R}}

\newtheorem{defn}{Definition}[section]
\newtheorem{thm}[defn]{Theorem}

\newtheorem{prop}[defn]{Proposition}

\newtheorem{lemma}[defn]{Lemma}

\title{$\Z^{2}$-dimension groups}

\author{Thierry Giordano\thanks{Supported in part by a
grant from NSERC, Canada}, \\
Department of Mathematics and Statistics, \\
University of Ottawa,\\
 STEM Complex, room 342 \\
150 Louis-Pasteur Pvt \\
Ottawa, ON, Canada K1N 6N5
 \and
 Ian F. Putnam\thanks{Supported in part by a
grant from NSERC, Canada},\\
Department of Mathematics and Statistics,\\
University of Victoria,\\
Victoria, B.C., Canada V8W 3R4
\and 
Christian F. Skau\thanks{Supported in
part by the Norwegian Research Council}, \\
Department of Mathematical Sciences, \\
Norwegian
University of Science and Technology (NTNU), \\
N-7034 Trondheim, Norway}

\date{}

\begin{document}
\maketitle

\begin{dedication}
{Dedicated to the memory of Anatoly Moiseevich Vershik
and his profound contributions to ergodic theory, dynamical systems, operator algebras
and representation theory}
\end{dedication}

\pagebreak 

\begin{abstract}
We study a class of simple dimension groups in which the cyclic
subgroup generated by the order unit is replaced
by a copy of $\Z^{2}$ satisfying some strict conditions.
Our main results are necessary and sufficient conditions
on a Bratteli diagram  which provides inductive limit structures
for  such groups. This result
has an important application in constructing a version of the 
Bratteli-Vershik model for minimal actions of $\Z^{2}$ on the Cantor set 
which
will be the subject of a subsequent paper.
\end{abstract}

\section{Introduction and statements of the results}

Dimension groups were first introduced in George Elliott's 
seminal paper classifying inductive limits of finite dimensional
$C^{*}$-algebras (or AF-algebras) by their scaled, ordered K-theory.
We refer the reader to Elliott's original paper \cite{Ell:AF},
Effros'  lecture notes \cite{Eff:CBMS}, the comprehensive 
books of Goodearl \cite{Good:book} and Durand and Perrin \cite{DP:book}.

An \emph{ordered abelian group}
 is a pair $(G, G^{+})$, where $G$ is an abelian group 
(here, these will always be countable) and a subset $G^{+}$ of $G$ 
satisfying \begin{enumerate}
\item $G^{+} + G^{+} \subseteq G^{+}$, 
\item $G^{+} - G^{+} = G$,
\item $G^{+} \cap (-G^{+}) = \{ 0 \}$.
\end{enumerate}
The subset $G^{+}$ is called the \emph{positive cone} for $G$.
(Sometimes, for K-theoretic considerations, the final axiom is dropped.)

We also recall that an element $u$ of $G^{+}$ is an \emph{order unit} if, 
for every $a$ in $G$, there is a positive integer $k$ such that 
$ku-a$ is in $G^{+}$. A \emph{positive group homomorphism} from 
$(G, G^{+})$ to $(H, H^{+})$ is a group homomorphism 
$\alpha:G \rightarrow H$ such that $\alpha(G^{+})\subseteq H^{+}$.
If $(G, G^{+})$ has an order unit $u$, then a \emph{state with respect to $u$} is a 
positive group homomorphism $\sigma$
 from $(G, G^{+})$ to $(\R, \R^{+})$ such that $\sigma(u) =1$. For the most part, 
 we will
 consider the case where a unique order unit has been specified 
 and simply refer to $\sigma$ as a state. The set of 
 all states is naturally a subset of $\R^{G}$ and is a compact, convex set. 
 A state is extremal if it is an extreme point of the convex set
 of states.

The principal examples of  countable ordered abelian 
groups are the free abelian 
groups with simplicial or standard order:
 take $I \geq 1$, the group $\Z^{I}$ with 
$\Z^{I+} = \{ (n(1), \ldots, n(I)) \mid n(1), \ldots, n(I) \geq 0 \}$.
 It is an easy exercise
to see that any group homomorphism $\varphi: \Z^{I} \rightarrow \Z^{J}$
arises as $\varphi(n) = An, n \in \Z^{I}$, where $A$ is a 
$J \times I$-matrix with integer entries
and the homomorphism is positive if and only if $A$ has non-negative entries.
Here, of course, we must regard $\Z^{I}$ as column vectors. Usually, 
we will use Roman letters for column vectors and Greek letters 
for row vectors.

We can construct inductive limits of these groups as follows:
begin with a sequence of positive integers, $I_{n}, n \geq 1,$ and 
non-negative integer matrices $A_{n}, n \geq 1$, where $A_{n}$ is 
of size
$I_{n+1} \times I_{n}$. We have the sequence
 \[
\xymatrix{
 \Z^{I_{1} }  \ar[r]^{A_{1}} & \Z^{I_{2} }  \ar[r]^{A_{2}} &
  \cdots \ar[r] &  \cdots } 
\]
This data can be encoded by a combinatorial object called a 
Bratteli diagram, which we denote by $(V,E)= (V_{n}, E_{n}), n \geq 1$
 \cite{Bra:BD,DP:book,Put:CMS}. 
The vertex set $V_{n} $ is $\{ 1, 2, \ldots, I_{n} \}$, while we
let the $j,i$-entry of
$A_{n}$, written $A_{n}(j,i)$,  be the number of edges 
in  $E_{n}$ from vertex $i$ in $V_{n}$ to vertex $j$ in $V_{n+1}$.

The inductive limit in the category of ordered abelian groups
of this may be described concretely: see \cite{Put:CMS}.
It is an ordered abelian group $(G, G^{+})$ and, for each $n \geq 0$, 
we have a canonical positive homomorphism 
$\gamma_{n}: \Z^{I_{n}} \rightarrow G$.

A dimension group is any ordered abelian group obtained by this
construction. We remark that it follows easily 
from the construction that such a group
is torsion-free. The remarkable Effros-Handelman-Shen Theorem 
(Theorem 2.4.3 of \cite{DP:book} or Theorem 8.4 of \cite{Put:CMS}) 
asserts that an ordered abelian group is a dimension group if and 
only if it is countable, unperforated and satisfies Riesz
interpolation. In fact, some treatments prefer this description
as the definition of dimension group.

The 'only if' direction is relatively simple: the specific properties
are satisfied by the ordered groups 
$(\Z^{I}, \Z^{I+})$ and are preserved under
inductive limits. The 'if' direction is deep and powerful.

Aside from simply giving the existence of inductive limit structures
for these groups, it also raises an important question: how are structural
properties of the inductive limit group manifested in the combinatorial 
properties of the inductive system? Given an inductive system
as above, does some property of the limit group yield combinatorial information
about the system? This type of question has three 
kinds of answers. The first is information directly about the given system. As 
an example, the lattice of order ideals in the group determines the 
lattice of hereditary subsets of the diagram. (We say 
$X \subseteq \cup_{n} V_{n}$ is hereditary if, for any $i$ in $X \cap V_{n}$ and 
$j$ in $V_{n+1}$, if $A_{n}(j,i) \neq 0$, then $j$ is in $X$ also.)
The second type of answer is that the system has some property, 
after telescoping \cite{DP:book}. For example, if the 
group is simple (i.e. has no order ideals), then we
can assume $A_{n}$ has strictly positive entries for all $n$, after telescoping.
The third type of answer is that the diagram is equivalent to one having some
property. Or put in another way, any group with a specific
property may be written as an inductive system of some type.
The main result of this paper is an answer of the third type.

Since we will always be dealing with simple dimension groups, let us
remark that a dimension group is simple if and only if every 
non-zero positive element  (called strictly positive) 
is an order unit (see  section 2.2.1  of \cite{DP:book}).

The dimension groups we will consider contain a canonical copy of 
$\Z^{2}$ satisfying some properties, hence the title of the paper.
One simple example is $ \Q^{2}$. More generally, one can consider
any countable dense subgroup of $\R^{2}$ which contains $\Z^{2}$.
We remark that, unlike the situation for 
$\Q$, there is no classification of subgroups
of $\Q^{2}$ \cite{Thom:class}.
Our starting point is  the following relatively simple result. 

\begin{thm}
\label{intro:10}
Let $G$ be a
countable, torsion-free abelian group and let 
$\sigma:G \rightarrow \R^{2}$ be a group homomorphism satisfying:
\begin{enumerate}
\item $\Z^{2} \subseteq \sigma(G)$, 
\item $\sigma(G) \subseteq \R^{2}$ is dense.
\end{enumerate}
If we define
\[
G^{+} =  \{ 0 \} \cup \sigma^{-1}((0, \infty) \times (0, \infty)),
\]
then $G$ is a simple dimension group.
Further, if $u^{1}$ and $u^{2}$  
are any elements of $G$ with $\sigma(u^{1}) = (1,0)$ 
and $\sigma(u^{2}) = (0,1)$,
 then neither is positive while 
$u= u^{1} + u^{2}$ is an order unit for $G$.

Letting $\sigma_{1}, \sigma_{2}:G \rightarrow \R$  be 
the two coordinate functions of $\sigma$; that is,
$\sigma(g) = (\sigma_{1}(g), \sigma_{2}(g))$, for all $g$ in $G$, then
$\sigma_{1}, \sigma_{2}$ are states on $(G, G^{+})$. Moreover, these are the
only extremal states on $G$.
\end{thm}

It is a fairly simple matter to show that $(G, G^{+})$ is unperforated 
and, as shown on page 16 of \cite{Eff:CBMS}, 
 satisfies Riesz interpolation (as a consequence of condition 2).
The  Effros-Handelman-Shen Theorem then asserts that $(G, G^{+})$ may be written as 
an inductive limit of groups $(\Z^{I}, \Z^{I+})$, with the usual simplicial (direct sum)
order. It is a simple matter to check that every 
strictly positive element
is an order unit, which implies the ordered group is simple.
Unfortunately, the proof of the Effros-Handelman-Shen
 Theorem is not really helpful in 
explicitly constructing
such an inductive system. This is the main topic of the paper.

Before going on, it is probably
 worth noting  a kind of converse to Theorem \ref{intro:10}. The result is almost
trivial but it does provide a good description of exactly which 
dimension groups arise in this context.

\begin{thm}
\label{intro:20}
Let $(G, G^{+})$ be a simple dimension group.
 Suppose that $(G, G^{+})$
has exactly two extremal states, $\sigma_{1}, \sigma_{2}$, and 
two elements $u^{1}, u^{2}$ such that 
\begin{eqnarray}
\sigma_{1}(u^{1})  & = \sigma_{2}(u^{2}) & = 1, \\
\sigma_{1}(u^{2})  & = \sigma_{2}(u^{1}) & = 0.
\end{eqnarray}
If we define $\sigma:G \rightarrow \R^{2}$ by 
$\sigma(g) = ( \sigma_{1}(g), \sigma_{2}(g))$, for 
$g$ in $G$, then  
\[
G^{+} = \{ 0 \} \cup \sigma^{-1}( (0, \infty) \times (0, \infty) ),
\]
$\Z^{2} \subseteq \sigma(G)$ and $\sigma(G)$ is dense in $\R^{2}$.
\end{thm}

The only non-trivial parts are the description given of $G^{+}$ and the 
fact that $\sigma(G)$ is dense.
The former follows easily from the fact that, in a simple dimension group
having an  order unit,
an element $g$ is positive if and only if $g=0$ or 
$\tau(g) > 0$, for every state $\tau$ on $(G, G^{+})$.
 In our case, the latter 
is equivalent to $\sigma_{1}(g), \sigma_{2}(g) > 0$.
As for the fact that  the image of $\sigma$ is dense, the values 
$\sigma_{1}(g)$ and $\sigma_{2}(g)$ uniquely determine $\tau(g)$, for
any state $\tau$. In a simple dimension group, the set of affine 
functions of the form $\tau \rightarrow \tau(g), g \in G$ are dense
in the space of affine functions on the space of states
 (see Theorem 4.4  \cite{Eff:CBMS}
 or Chapter 8 of \cite{Good:book}).

We remark that the topic of dimension groups
having two extremal states began with the work of 
Fack and Marechal  \cite{FM:2states}.  The distinction with our work
is that we assume the group has elements which lie
on the 'boundary' of the positive cone, namely, $u^{1}, u^{2}$.
This hypothesis is a strong restriction on the  group and these
form a very special subclass.
Let us further remark that if $G$ is a dense subgroup of
$\R^{2}$ and $P$ is an open positive cone in $\R^{2}$ bounded
by two rays that both meet $G$, then a simple linear transformation
shows that our results can be applied to $(G, G^{+}= G\cap P)$, using
$\sigma$ to be the inclusion map.

 If $v$ is a vector in $\R^{I}$ and $r$ is a real number, 
 we write $v \leq r$ (or $v \geq r$) if $v(i) \leq r$
  for all $1 \leq i \leq I$ (or $v(i) \geq r$
  for all $1 \leq i \leq I$, respectively). For any partially 
  ordered set with elements $a,b,c,d$, we will write 
  $a \leq b, c \leq d$ to mean $a \leq b \leq d$ and $a \leq c \leq d$.

\begin{thm}
\label{intro:30}
Suppose we are given
\begin{enumerate}
\item a sequence of  integers $J_{n}, n \geq 1$,
\item $u^{1,1}_{n}, u^{1,2}_{n}, u^{2,1}_{n}, u_{n}^{2,2}$ in $\Z^{J_{n}}$,
for all $n \geq 1$,
\item strictly positive integer $J_{n+1} \times J_{n}$-matrices, 
$A^{1,1}_{n}, A^{1,2}_{n}, A^{2,1}_{n}, A_{n}^{2,2}$, for all $n \geq 1$,
satisfying 
\[
 \left[ \begin{matrix}
 A_{n}^{1,1} & A_{n}^{1,2} \\A_{n}^{2,1} & A_{n}^{2,2} \end{matrix} \right]
 \left[ \begin{matrix}
 u_{n}^{i,1} \\ u_{n}^{i,2}  \end{matrix} \right]
 = \left[ \begin{matrix}
 u_{n+1}^{i,1} \\ u_{n+1}^{i,2}  \end{matrix} \right],
 \]
 for $i=1,2$.
Let $(G, G^{+})$ be  the dimension group
which is the inductive limit
\[
\xymatrix{
 \Z^{J_{1} } \oplus    \Z^{J_{1} } \ar[rrr]<2pt>^{\left[ \begin{matrix}
 A_{1}^{1,1} & A_{1}^{1,2} \\
 A_{1}^{2,1} & A_{1}^{2,2} \end{matrix} \right]} & & &
 \Z^{J_{2} } \oplus    \Z^{J_{2} } 
 \ar[rrr]^{\left[ \begin{matrix}
 A_{2}^{1,1} & A_{2}^{1,2} \\
 A_{2}^{2,1} & A_{2}^{2,2} \end{matrix} \right]} & & &
  \cdots  &   } 
\]
with $\eta_{n}: \Z^{J_{n}} \oplus \Z^{J_{n}} \rightarrow G$, 
the natural inclusion, for each $n \geq 1$ and let
$u^{i} = \eta_{n}(u_{n}^{i,1}, u_{n}^{i,2})$, for $i=1,2$. 

 Suppose that there are 
 a decreasing sequence of rational numbers $2 > r_{n} > 1, n \geq 1$,
  with limit $1$
 and strictly 
 increasing sequences of positive integers $K_{n}$ and 
 $ l_{n}$, both at least $ 4$,
 such that
 \item 
 \begin{eqnarray*}
 r_{n}^{-1} K_{n} & \leq u_{n}^{1,1}, u_{n}^{2,2}
  \leq & r_{n}K_{n}, \\
  -r_{n}K_{n} & \leq l_{n}u_{n}^{1,2}, l_{n}u_{n}^{2,1}
   \leq & - r_{n}^{-1} K_{n},
 \end{eqnarray*}
 \item 
 \begin{eqnarray*}
 r_{n}^{-1} A_{n}^{1,1}(j',j) & \leq l_{n} A_{n}^{2,1}(j'',j)
  \leq & r_{n} A_{n}^{1,1}(j',j), 
 \\
 r_{n}^{-1} A_{n}^{2,2}(j',j) & \leq l_{n} A_{n}^{1,2}(j'',j)
  \leq & r_{n} A_{n}^{2,2}(j',j), 
 \end{eqnarray*}
 for all $1 \leq j',j'' \leq J_{n+1}, 1 \leq j \leq J_{n}$. 
  \end{enumerate}
  
Then $(G, G^{+})$ is a 
simple dimension group with exactly two extremal
 states $\sigma_{1}, \sigma_{2}$ satisfying
 $\sigma_{1}(u^{1}) = \sigma_{2}(u^{2})=1$ and 
 $\sigma_{2}(u^{1}) = \sigma_{1}(u^{2})=0$.
\end{thm}

Our main goal is the following converse of this last theorem.

\begin{thm}
\label{intro:40}
Let $(G, G^{+})$ be a simple dimension group with elements \newline
$u^{1}, u^{2}$
and exactly two extremal states, $\sigma_{1}, \sigma_{2}$, with 
$\sigma_{1}(u^{1}) = \sigma_{2}(u^{2}) =1 $ and 
$\sigma_{2}(u^{1}) = \sigma_{1}(u^{2}) =0$.
Let $a_{n}, n = 1, 2, \ldots$ be a  sequence of positive integers.

Then there exist:
\begin{enumerate}
\item a sequence of positive integers $J_{n}, n \geq 1$,
\item $u^{1,1}_{n}, u^{1,2}_{n}, u^{2,1}_{n}, u_{n}^{2,2}$ in $\Z^{J_{n}}$,
for all $n \geq 1$,
\item strictly 
positive integer $J_{n+1} \times J_{n}$-matrices, 
$A^{1,1}_{n}, A^{1,2}_{n}, A^{2,1}_{n}, A_{n}^{2,2}$, for all $n \geq 1$,
such that the inductive limit
\[
\xymatrix{
 \Z^{J_{1} } \oplus    \Z^{J_{1} } \ar[rrr]<2pt>^{\left[ \begin{matrix}
 A_{1}^{1,1} & A_{1}^{1,2} \\
 A_{1}^{2,1} & A_{1}^{2,2} \end{matrix} \right]} & &  &
 \Z^{J_{2} } \oplus    \Z^{J_{2} } 
 \ar[rrr]^{\left[ \begin{matrix}
 A_{2}^{1,1} & A_{2}^{1,2} \\
 A_{2}^{2,1} & A_{2}^{2,2} \end{matrix} \right]} & & & 
  \cdots    } 
\]
is isomorphic to $(G, G^{+})$ as ordered abelian groups.
If we let $\eta_{n} $
 denote the resulting natural
map from $\Z^{J_{n} } \oplus    \Z^{J_{n} }$ to $G$, then 
$\eta_{n}(u_{n}^{i,1}, u_{n}^{i,2}) 
=u^{i}$ and 
\[
 \left[ \begin{matrix}
 A_{n}^{1,1} & A_{n}^{1,2} \\
 A_{n}^{2,1} & A_{n}^{2,2} \end{matrix} \right]
 \left[ \begin{matrix}
 u_{n}^{i,1} \\ u_{n}^{i,2}  \end{matrix} \right]
 = \left[ \begin{matrix}
 u_{n+1}^{i,1} \\ u_{n+1}^{i,2}  \end{matrix} \right]
 \]
 for $i=1,2$.
 
 In addition, there exist, for all $n \geq 1 $, a 
 decreasing sequence of rational numbers 
 $2 > r_{n} > 1, n \geq 1$, with limit $1$ and 
 increasing sequences of positive 
integers $K_{n}, l_{n}, M_{n}, n \geq 1$, all at least $ 4$, 
such that

 \item 
 \begin{eqnarray*}
 r_{n}^{-1} K_{n} & \leq u_{n}^{1,1}, u_{n}^{2,2} 
 \leq & r_{n}K_{n}, \\
  -r_{n} K_{n} & \leq l_{n}u_{n}^{1,2}, l_{n}u_{n}^{2,1} 
  \leq & - r_{n}^{-1}  K_{n}
 \end{eqnarray*}
 \item 
 \begin{eqnarray*} a_{n} &  \leq  
  A_{n}^{1,2}(j',j), &  
 A_{n}^{2,1}(j',j),    \\
 r_{n}^{-1} A_{n}^{1,1}(j',j) & \leq 
 l_{n} A_{n}^{2,1}(j'',j)
  \leq & r_{n} A_{n}^{1,1}(j',j), 
 \\
 r_{n}^{-1} A_{n}^{2,2}(j',j) & \leq l_{n} A_{n}^{1,2}(j'',j)
  \leq & r_{n} A_{n}^{2,2}(j',j), 
 \end{eqnarray*}
 for all $1 \leq j',j'' \leq J_{n+1}, 1 \leq j \leq J_{n}$,
 \item 
  \begin{eqnarray*}
 r_{n}^{-1} M_{n} & \leq \sum_{j=1}^{J_{n}} A_{n}^{1,1}(j',j),
 \sum_{j=1}^{J_{n}} A_{n}^{2,2}(j',j)
  \leq & r_{n} M_{n},  
 \end{eqnarray*}
 for all $1 \leq j' \leq J_{n+1}$.
\end{enumerate}
\end{thm}

Notice that this statement is actually
stronger  than the converse of Theorem \ref{intro:30} because
of condition 6 and the conclusion involving the sequence 
$a_{n}$.
 This is something we will need in our application.

We believe this result is of interest in its own right, but we 
would like to explain our main application, which is 
to topological dynamical systems and will appear 
in a subsequent paper, now in preparation. 

 By now, the connection
between dimension groups and dynamical systems is well-established
(for example, see Durand and Perrin \cite{DP:book}) and that connection
goes through Bratteli diagrams. By the late 1970's,
several people were working along these lines, including
Stratila and Voiculescu \cite{SV:Uinfin}, Krieger \cite{Kri:paper}
 and Renault 
\cite{Ren:LNM}. But
Anatoly Vershik's notion of
an 'adic' transformation on the path space of a Bratteli diagram
(\cite{Ver:adic} and \cite{Ver:Mark})
was a crucial idea which eventually led to Herman, Putnam and Skau
providing the Bratteli-Vershik model for minimal 
actions of the group of integers on the Cantor set \cite{HPS:model}. The 
key point is that the Bratteli diagram at once describes both the 
space with the action and also  a cohomological invariant
in the form of the dimension group. This, in turn, led to 
a classification of such systems up to orbit equivalence in terms
of an invariant closely linked with the dimension group
by Giordano, Putnam and Skau \cite{GPS:orb}.

While this classification up to orbit equivalence 
has been extended to include minimal free
actions of $\Z^{d}$ by Giordano, Matui, Putnam and Skau,
 first for $d  =2$ \cite{GMPS:orb2} and later for 
$d > 2$ \cite{GMPS:orbd}, the model, even for actions of 
$\Z^{2}$, has not. There have been interesting 
attempts by Vershik and Lodkin \cite{VerL:Zd} who developed an analogue
of adic transformations and by Frank and Sadun 
('fusion rules') \cite{FS:fusion}, but
none carries the cohomological data in a good way. 
 Our goal is to provide a model, at least in the case $d=2$,
  for minimal free actions of
 $\Z^{2}$ on the Cantor set, which take as their input
 cohomological data. In other words, to find such dynamical 
 systems with  prescribed cohomology. We remark that, 
 in the case $d=1$, the model is complete in the sense that every 
 minimal action of the integers is realized. We do not aim
 for such a general result. Indeed, there are 
 examples of Clark and Sadun  \cite{CS:cohom} whose cohomology do 
 not satisfy the conditions we need for our model.
 
 The cohomology, for a given action $\varphi$ of $\Z^{d}$ on 
 a Cantor set $X$, is the group cohomology of the acting group, $\Z^{d}$,
 with coefficients in $C(X, \Z)$, the continuous 
 integer-valued functions on $X$. We will not go into much detail
 but we denote the result by
  $H^{*}(X, \varphi)$. Under our hypotheses, we always have 
  $H^{0}(X, \varphi)  \cong \Z$, while  $H^{k}(X, \varphi) =0$ for
  $k >d$. In addition,  a quotient of 
  $H^{d}(X, \varphi)  $ is the invariant for orbit equivalence. 
  Of course, when $d=1$, that is the end of the story. For $d > 1$,
  it is our opinion that the group
   $H^{1}(X, \varphi)  $ is very important. Some evidence for
   this is given in the case of odometer actions of $\Z^{d}$, as described 
   by Giordano, Putnam and Skau \cite{GPS:odom}.
   
    It is a simple
   matter to see that $H^{1}(X, \varphi)  $  is countable and 
   torsion-free (see Theorem 6.1 of \cite{GPS:odom}).
   A $1$-cocycle 
   is represented by a continuous
    function $\theta: X \times \Z^{d} \rightarrow \Z$. Then if 
    $\mu$ is any $\varphi$-invariant measure on $X$, by simply 
    integrating out the $x$ variable, we obtain a group homomorphism
    from   $H^{1}(X, \varphi)  $ to $Hom(\Z^{d}, \R)$, which we
    can identify with $\R^{d}$. Moreover, the range contains
    $\Z^{d}$. Our goal then is to begin with a countable torsion-free
    abelian group $G$ along with a homomorphism to $\R^{2}$
    whose image contains $\Z^{2}$. We add 
    the hypothesis that the range is dense and construct a 
    minimal free action, $\varphi$, of $\Z^{2}$ on a Cantor set $X$
   with $H^{1}(X, \varphi) \cong G$. We view this 
   as the Bratteli-Vershik model for $\Z^{2}$-actions. 
   The current paper 
   begins with an abstract dimension group of a specific form 
   as above and 
   produces the combinatorial data in the form of a Bratteli diagram
   as the first step in the construction of the dynamics.

The authors are grateful to the referee for a thorough
 reading of the paper and numerous helpful corrections and 
 comments. Part of this research was conducted during a Research in Pairs
 visit to the Mathematisches Forschungsinstitut Oberwolfach and during 
 a Research in Teams visit to the Banff International Research Station.
The authors gratefully acknowledge these supports.

\section{Sufficiency}

In this section, we prove Theorem \ref{intro:30}. We assume 
throughout that we have 
$A_{n}^{i,j}, u_{n}^{i,j}, r_{n}, K_{n}, l_{n}$,
 for $n \geq 1, i,j=1,2$
as in the statement of \ref{intro:30}. We also let 
$u_{n}^{i} = (u_{n}^{i,1}, u_{n}^{i,2})$,  and 
$u_{n} = u_{n}^{1}+ u_{n}^{2}$,
for all $n \geq 1$ and 
$i=1,2$.

Let us quickly observe that 
$l_{n} \geq 4  > r_{n}^{2}$, for all $n \geq 1$.

We first want to show that  $(G, G^{+})$ has at least two distinct
states which satisfy the desired conditions.

\begin{lemma}
\label{suff:10}
\begin{enumerate}
\item 
$(G, G^{+})$ is a simple dimension group.
\item 
$u= u^{1} + u^{2}$ is an order unit for $(G, G^{+})$.
\item 
For any state $\sigma$ on $G$, $0 \leq \sigma(u^{1}), \sigma(u^{2})$.
\item 
 There exist states $\sigma_{1}, \sigma_{2}$ on $(G, G^{+})$ such that 
 \begin{eqnarray*}
\sigma_{1}(u^{1})  & = \sigma_{2}(u^{2}) & = 1, \\
\sigma_{1}(u^{2})  & = \sigma_{2}(u^{1}) & = 0.
\end{eqnarray*}
\end{enumerate}
\end{lemma}

\begin{proof}
The fact that $(G, G^{+})$ is a simple dimension group
follows from the fact that the matrices are strictly positive.

From condition 4 of Theorem \ref{intro:30} and the fact that 
$l_{n} \geq 4 > r_{n}^{2}$,  $u_{n}$ is a 
strictly positive element of $\Z^{J_{n}}\oplus \Z^{J_{n}}$ for all 
$n$. The second statement follows from the fact that $(G, G^{+})$ is 
simple.

We observe that the estimates from part four of Theorem \ref{intro:30}
 show that,
 for any $n \geq 1$, 
 \begin{eqnarray*}
 l_{n}u_{n}^{1} + 4u_{n}^{2}
  & = & (l_{n}u_{n}^{1,1} + 4u_{n}^{2,1}, l_{n}u_{n}^{1,2} +4 u_{n}^{2,2})
  \\
  & \geq & (l_{n} r_{n}^{-1} K_{n} - 4r_{n}l_{n}^{-1} K_{n}, 
   - r_{n}K_{n} + 4r_{n}^{-1} K_{n} ) \\
     &  = & K_{n} ( l_{n}r_{n}(r_{n}^{-2} - 4 l_{n}^{-2}), 
   r_{n}^{-1} (4- r_{n}^{2} )) \\
     &  \geq  &  K_{n}( l_{n}r_{n}(4^{-1} - 4 \cdot 16^{-1}), 
     r_{n}^{-1}(4 - 4)) \\
       &  = & (0,0).
\end{eqnarray*} 
 It follows that if $\sigma$ is any state and $n \geq 1$,
 \[
 0 \leq \sigma \circ \eta_{n}( l_{n}u_{n}^{1} + 4u_{n}^{2} )
  = l_{n} \sigma(u^{1}) + 4\sigma(u^{2}).
  \]
 As $n $ is arbitrary and $l_{n}$ is unbounded, 
 $\sigma(u^{1}) \geq 0$. A similar argument shows
 that $\sigma(u^{2}) \geq 0$.
 
 For the last part, it is clear from the fact that
  $u_{n}^{1,2}$ is 
 negative for all $n \geq 1$ that 
 $u^{1}$ is not positive. It follows that there 
 exists a state, $\sigma_{2}$, 
 such that 
 $\sigma_{2}(u^{1}) \leq 0$. Similarly, there is a state, $\sigma_{1}$, 
 such that 
 $\sigma_{1}(u^{2}) \leq 0$. Combining this with the last statement 
implies $\sigma_{2}(u^{1}) = \sigma_{1}(u^{2}) =0$. Hence, we have
$1 = \sigma_{1}(u) = \sigma_{1}(u^{1}+u^{2})= \sigma_{1}(u^{1})$. 
Similarly, we have $\sigma_{2}(u^{2})=1$. 
\end{proof}

We now want to prove that our two states are exactly the extremal states
for the dimension group $(G, G^{+})$. We begin with a technical result.

\begin{lemma}
\label{suff:20}
Let $n \geq 1$ and   $x^{1}, x^{2}$ be in $\Z^{J_{n}}$ 
satisfying $(0,0) \leq (x^{1}, x^{2}) \leq C u_{n}$ in 
$\Z^{J_{n}} \oplus \Z^{J_{n}}$, for some positive number $C$.
 Define 
\[
  \left[ \begin{matrix}
 y^{1} \\ y^{2}  \end{matrix} \right] =
 \left[ \begin{matrix}
 A_{n}^{1,1} & A_{n}^{1,2} \\A_{n}^{2,1} & A_{n}^{2,2} \end{matrix} \right]
 \left[ \begin{matrix}
 x^{1} \\ x^{2}  \end{matrix} \right].
 \]
There exist positive real numbers $s, t, \bar{s}, \bar{t}$ such that 
\[
   r_{n}^{-2}\left( su_{n+1}^{1} + tu_{n+1}^{2}  \right)
  \leq (y^{1}, y^{2}) \leq  
  r_{n}^{2}\left(  \bar{s}u_{n+1}^{1} + \bar{t}u_{n+1}^{2}  \right)
\]
in $\R^{J_{n+1}} \oplus \R^{J_{n+1}} $ and 
\begin{eqnarray*}
 0 \leq & (r_{n+1}^{2} - r_{n+1}^{-2}l_{n+1}^{-2}) 
 \bar{s} - (r_{n+1}^{-2} - r_{n+1}^{2}l_{n+1}^{-2})s   &  
 \leq C r_{n}^{2}r_{n+1} (r_{n+1}-r_{n+1}^{-1}), \\
  0 \leq & (r_{n+1}^{2} - r_{n+1}^{-2}l_{n+1}^{-2}) 
 \bar{t} - (r_{n+1}^{-2} - r_{n+1}^{2}l_{n+1}^{-2})t   &  
 \leq C r_{n}^{2}r_{n+1}  (r_{n+1}-r_{n+1}^{-1}).
 \end{eqnarray*}
\end{lemma}

\begin{proof}
First, for $ i = 1,2 $, 
define $X^{i} = \sum_{j=1}^{J_{n}} A_{n}^{i,i}(1,j) x^{i}(j)$, 
which is clearly non-negative.
Then, for any $1 \leq j' \leq J_{n+1}$, we have
\begin{eqnarray*}
y^{1}(j') & = & \sum_{j=1}^{J_{n}} \left( A^{1,1}_{n}(j', j) x^{1}(j) + 
 A^{1,2}_{n}(j', j) x^{2}(j) \right) \\
  & \leq & \sum_{j=1}^{J_{n}} \left(
   r_{n} l_{n}A^{2,1}_{n}(j', j) x^{1}(j) + 
 A^{1,2}_{n}(j', j) x^{2}(j)\right) \\
   & \leq & \sum_{j=1}^{J_{n}} \left(
    r_{n}^{2} A^{1,1}_{n}(1, j) x^{1}(j) + 
   r_{n} l_{n}^{-1}  A^{2,2}_{n}(1, j) x^{2}(j)\right) \\
  &  \leq & r_{n}^{2}( X^{1} + l_{n}^{-1} X^{2}).
  \end{eqnarray*}
  
 Similar calculations also prove
  \begin{eqnarray*}
y^{2}(j') 
  &  \leq & r_{n}^{2}(l_{n}^{-1}  X^{1} + X^{2}), \\
  y^{1}(j') 
  &  \geq & r_{n}^{-2}(  X^{1} + l_{n}^{-1} X^{2}), \\
  y^{2}(j') 
  &  \geq & r_{n}^{-2}(l_{n}^{-1}  X^{1} + X^{2}). \\
  \end{eqnarray*}
  
  All four  inequalities  can be summarized as follows
  \[
  r_{n}^{-2} \left[ \begin{matrix}
 1 & l_{n}^{-1} \\ l_{n}^{-1}  & 1 \end{matrix} \right]
 \left[ \begin{matrix}
 X^{1} \\ X^{2}  \end{matrix} \right]
 \leq  
  \left[\begin{matrix} y^{1} \\ y^{2} \end{matrix} \right]
\leq  r_{n}^{2} 
 \left[ \begin{matrix}
 1 & l_{n}^{-1} \\ l_{n}^{-1}  & 1 \end{matrix} \right]
 \left[ \begin{matrix}
 X^{1} \\ X^{2}  \end{matrix} \right].
 \]
  
  We also have 
  \[
   \left[ \begin{matrix}
 y^{1} \\ y^{2}  \end{matrix} \right] =
 \left[ \begin{matrix}
 A_{n}^{1,1} & A_{n}^{1,2} \\A_{n}^{2,1} & A_{n}^{2,2} \end{matrix} \right]
 \left[ \begin{matrix}
 x^{1} \\ x^{2}  \end{matrix} \right]
 \leq C
 \left[ \begin{matrix}
 A_{n}^{1,1} & A_{n}^{1,2} \\A_{n}^{2,1} & A_{n}^{2,2} \end{matrix} \right]
 \left[ \begin{matrix}
 u_{n}^{1} \\ u_{n}^{2}  \end{matrix} \right]
 = C
 \left[ \begin{matrix}
 u_{n+1}^{1} \\ u_{n+1}^{2}  \end{matrix} \right]
  \]
  which then implies
  \[
 \left[ \begin{matrix}
 1 & l_{n}^{-1} \\ l_{n}^{-1}  & 1 \end{matrix} \right]
 \left[ \begin{matrix}
 X^{1} \\ X^{2}  \end{matrix} \right] \leq 
 r_{n}^{2} \left[\begin{matrix} y^{1} \\ y^{2} \end{matrix} \right]
   \leq r_{n}^{2} C  \left[\begin{matrix} u_{n+1}^{1} \\ u_{n+1}^{2} \end{matrix} \right] 
   \leq r_{n}^{2} r_{n+1} C  \left[\begin{matrix} K_{n+1} \\ K_{n+1} \end{matrix} \right]
 \]  
 with the last inequality given 
  as condition 4 in Theorem \ref{intro:30}.

  For any positive real numbers $s, t$,  $n \geq 1$ 
  and $1 \leq j, j' \leq J_{n+1}$,
   we have 
  \begin{eqnarray*}
  su_{n+1}^{1,1}(j) + t u_{n+1}^{2,1}(j) & \leq & 
  sr_{n+1} K_{n+1}- t r_{n+1}^{-1} l_{n+1}^{-1} K_{n+1} \\
  &  =&    K_{n+1}( r_{n+1} s - r_{n+1}^{-1} l_{n+1}^{-1} t).
   \end{eqnarray*}
   
   Similar calculations also prove
  \begin{eqnarray*}
 su_{n+1}^{1,2}(j) + t u_{n+1}^{2,2}(j) 
  &  \leq  &    K_{n+1}( - r_{n+1}^{-1} l_{n+1}^{-1} s + r_{n+1}  t), \\
   su_{n+1}^{1,1}(j) + t u_{n+1}^{2,1}(j) & \geq & 
     K_{n+1}( r_{n+1}^{-1} s - r_{n+1} l_{n+1}^{-1} t), \\
      su_{n+1}^{1,2}(j) + t u_{n+1}^{2,2}(j) 
  &   \geq & 
     K_{n+1}( - l_{n+1}^{-1} r_{n+1} s +  r_{n+1}^{-1}  t). 
  \end{eqnarray*}
  
   These can be summarized in the form
  \begin{eqnarray*}
   su_{n+1}^{1} + t u_{n+1}^{2} & \leq   &  K_{n+1} \left[ \begin{matrix}
  r_{n+1} &  - r_{n+1}^{-1} l_{n+1}^{-1} \\ 
   - r_{n+1}^{-1} l_{n+1}^{-1} &  r_{n+1} \end{matrix} \right]
    \left[ \begin{matrix}
 s \\ t \end{matrix} \right], \\
    su_{n+1}^{1} + t u_{n+1}^{2}  
  & \geq  &   K_{n+1} \left[ \begin{matrix}
  r_{n+1}^{-1} &  - r_{n+1} l_{n+1}^{-1} 
  \\  - r_{n+1} l_{n+1}^{-1} &  r_{n+1}^{-1} \end{matrix} \right] 
  \left[ \begin{matrix}
 s \\ t \end{matrix} \right].   
 \end{eqnarray*}
  
 We recall that $l_{n+1} \geq 4$ while $2 > r_{n+1}$
 and it follows that 
 $r_{n+1}^{-2} - r_{n+1}^{2} l_{n+1}^{-2} > 4^{-1} - 4 \cdot 4^{-2} = 0$.
   We define
 \[
  \left[ \begin{matrix}
 s \\ t  \end{matrix} \right]  = K_{n+1}^{-1} 
 (  r_{n+1}^{-2}  - r_{n+1}^{2} l_{n+1}^{-2})^{-1} 
  \left[ \begin{matrix}
 r_{n+1}^{-1} & r_{n+1}l_{n+1}^{-1} \\ 
 r_{n+1}l_{n+1}^{-1}  &   r_{n+1}^{-1} \end{matrix} \right]
  \left[ \begin{matrix}
 1 & l_{n}^{-1} \\ l_{n}^{-1}  & 1 \end{matrix} \right]
 \left[ \begin{matrix}
 X^{1} \\ X^{2}  \end{matrix} \right]
 \]
and 
 \[
  \left[ \begin{matrix}
 \bar{s} \\ \bar{t}  \end{matrix} \right]  = K_{n+1}^{-1} 
 (  r_{n+1}^{2}  - r_{n+1}^{-2} l_{n+1}^{-2})^{-1} 
  \left[ \begin{matrix}
 r_{n+1} & r_{n+1}^{-1}l_{n+1}^{-1} \\ 
 r_{n+1}^{-1}l_{n+1}^{-1}  &   r_{n+1} \end{matrix} \right]
  \left[ \begin{matrix}
 1 & l_{n}^{-1} \\ l_{n}^{-1}  & 1 \end{matrix} \right]
 \left[ \begin{matrix}
 X^{1} \\ X^{2}  \end{matrix} \right]
 \]
 noting that all are positive.
The first pair of 
inequalities of the conclusion are
 immediate consequences. It remains to verify
the last statement.

We have 
\begin{eqnarray*} 
  &  (r_{n+1}^{2} - r_{n+1}^{-2}l_{n+1}^{-2}) \left[ \begin{matrix}
 \bar{s} \\ \bar{t}   \end{matrix} \right]  -  
   (r_{n+1}^{-2} - r_{n+1}^{2}l_{n+1}^{-2}) \left[ \begin{matrix}
 s \\ t   \end{matrix} \right]  &   \\
  =  &  K_{n+1}^{-1} 
  \left[ \begin{matrix}
 r_{n+1} - r_{n+1}^{-1} & (r_{n+1}^{-1} - r_{n+1}) l_{n+1}^{-1} \\ 
 (r_{n+1}^{-1} - r_{n+1}) l_{n+1}^{-1}  &   r_{n+1} - r_{n+1}^{-1} \end{matrix} \right]
  \left[ \begin{matrix}
 1 & l_{n}^{-1} \\ l_{n}^{-1}  & 1 \end{matrix} \right]
 \left[ \begin{matrix}
 X^{1} \\ X^{2}  \end{matrix} \right]  &    \\
   =  &   K_{n+1}^{-1} 
 (r_{n+1} - r_{n+1}^{-1}) \left[ \begin{matrix}
  1 & -l_{n+1}^{-1} \\ 
 - l_{n+1}^{-1}  &  1 \end{matrix} \right]
  \left[ \begin{matrix}
 1 & l_{n}^{-1} \\ l_{n}^{-1}  & 1 \end{matrix} \right]
 \left[ \begin{matrix}
 X^{1} \\ X^{2}  \end{matrix} \right]  &    \\
  \leq   &   K_{n+1}^{-1} 
 (r_{n+1} - r_{n+1}^{-1}) 
  \left[ \begin{matrix}
 1 & l_{n}^{-1} \\ l_{n}^{-1}  & 1 \end{matrix} \right]
 \left[ \begin{matrix}
 X^{1} \\ X^{2}  \end{matrix} \right]  &    \\
  =  &   K_{n+1}^{-1} 
 (r_{n+1} - r_{n+1}^{-1}) 
  r_{n}^{2} r_{n+1} C 
 \left[ \begin{matrix}
 K_{n+1} \\ K_{n+1}  \end{matrix} \right].  &    
 \end{eqnarray*}
 On the other hand, the matrix
 \[
 \left[ \begin{matrix}
  1 & -l_{n+1}^{-1} \\ 
 - l_{n+1}^{-1}  &  1 \end{matrix} \right]
  \left[ \begin{matrix}
  1 & l_{n}^{-1} \\ 
  l_{n}^{-1}  &  1  \end{matrix} \right]
 = \left[ \begin{matrix}
 1 - l_{n}^{-1}l_{n+1}^{-1} & l_{n}^{-1}- l_{n+1}^{-1}
  \\ l_{n}^{-1} - l_{n+1}^{-1} & 1 - l_{n}^{-1}l_{n+1}^{-1}
   \end{matrix} \right] 
 \] 
 is clearly positive since $l_{n+1} \geq l_{n}$ so 
 all quantities are positive. This completes the proof.
\end{proof}

\begin{lemma}
\label{suff:30}
If $\sigma$ is any state on $(G, G^{+})$, then
 $\sigma = \sigma(u^{1}) \sigma_{1} + \sigma(u^{2})\sigma_{2}$. In particular,
 the only extremal states on $(G, G^{+})$ are $\sigma_{1}, \sigma_{2}$.
\end{lemma}

\begin{proof}
If suffices to let $x$ be any positive element of $G$
and show that 
\[
\sigma(x) = \sigma(u^{1}) \sigma_{1}(x) + \sigma(u^{2}) \sigma_{2}(x).
\]
 
 Without loss of generality, we may assume 
 $x= \eta_{n_{0}}(x^{1}_{n_{0}}, x^{2}_{n_{0}})$ for some
  $ 0 \leq (x^{1}_{n_{0}}, x^{2}_{n_{0}}) \leq C u_{n_{0}}$,
  $n_{0}$ and $C > 0$. For each $n \geq n_{0}$, we let
  $(x^{1}_{n}, x^{2}_{n})$ be the element 
  of $\Z^{J_{n}+} \oplus \Z^{J_{n}+}$ obtained from
   $(x^{1}, x^{2})$ in the 
  inductive system. The hypotheses of Lemma \ref{suff:20} hold for all
  $n \geq n_{0}$ and hence we 
  obtain $s_{n}, \bar{s}_{n}, t_{n}, \bar{t}_{n}$
   satisfying the conclusion of \ref{suff:20}. In particular, we have 
  \[
   r_{n}^{-2}\left( s_{n}u_{n+1}^{1} + t_{n}u_{n+1}^{2}  \right)
  \leq (x_{n+1}^{1}, x_{n+1}^{2}) \leq  
  r_{n}^{2}\left(  \bar{s}_{n}u_{n+1}^{1} + \bar{t}_{n}u_{n+1}^{2}  \right)
\]
 for all $n \geq n_{0}$. We apply
  $\sigma_{1} \circ \gamma_{n}, \sigma_{2}\circ \gamma_{n}$
   and then $\sigma\circ \gamma_{n}$ to obtain  
\begin{eqnarray*}
r_{n}^{-2} s_{n} & \leq \sigma_{1}(x) \leq & r_{n}^{2} \bar{s}_{n}, \\
r_{n}^{-2} t_{n} & \leq \sigma_{2}(x) \leq & r_{n}^{2} \bar{t}_{n}, \\
\end{eqnarray*}
 and 
 \[
 r_{n}^{-2}(s_{n}  \sigma(u^{1})+ t_{n}  \sigma(u^{2})) 
 \leq \sigma(x) \leq  
 r_{n}^{2}(\bar{s}_{n} \sigma(u^{1})+ \bar{t}_{n} \sigma(u^{2})).
 \]
 From the last part of the conclusion of Lemma 
 \ref{suff:20}, we also know that 
 $\bar{s}_{n} - s_{n}, \bar{t}_{n}-t_{n}$ tend to zero so the first
 two inequalities above (and the fact that $r_{n}$ tends to $1$)
 imply that $s_{n}, \bar{s}_{n}$ and 
 $t_{n}, \bar{t}_{n}$ converge
 to $\sigma_{1}(x) $ and $\sigma_{2}(x)$, respectively.
 Taking limits in the third inequality proves that 
 \[
 \sigma(x) = \sigma(u^{1}) \sigma_{1}(x) + \sigma(u^{2}) \sigma_{2}(x) 
 \]
 as desired.
\end{proof}

Theorem \ref{intro:30} follows from Lemmas \ref{suff:10} and 
\ref{suff:30}.

\section{Necessity}

In this section, we prove Theorem \ref{intro:40}.

We begin with $G, G^{+}, u^{1}, u^{2}, \sigma_{1}, \sigma_{2}$
and $\sigma$ as in Theorem \ref{intro:20}.

 We may write $(G, G^{+})$ as an inductive limit
 in the category of ordered abelian groups
 \[
\xymatrix{
 \Z^{I_{1} }  \ar[r]^{A_{1}} & \Z^{I_{2} }  \ar[r]^{A_{2}} &
  \cdots \ar[r] & G, } 
\]
where each $\Z^{I_{m}}$ is considered as column vectors
given the standard order and 
$A_{m}$ is a $I_{m+1} \times I_{m}$ non-negative integer matrix.
We let $\gamma_{m}$ denote the canonical positive homomorphism 
of $\Z^{I_{m}}$ into $G$.
We will let 
$\varepsilon_{i}, 1 \leq i \leq I,$ denote the standard 
generators of $\Z^{I}$.

By standard arguments, we may assume that the Bratteli diagram 
which provides the inductive system above
 has no sources, other than those in the first vertex set,
  or sinks. This implies
that, for each $m \geq 1 $ and $1 \leq i \leq I_{m}$, 
$\gamma_{m}(\varepsilon_{i})$ is strictly positive in $G$; that is, 
it lies in $G^{+} - \{ 0 \}$. We call such 
a map \emph{strictly positive}.

Without loss of generality, we may assume 
that, for each $m \geq 1$, there are 
elements $u_{m}^{1}, u_{m}^{2}$ satisfying 
$A_{m}u_{m}^{1} = u_{m+1}^{1}, A_{m}u_{m}^{2} = u_{m+1}^{2},
 \gamma_{m}(u_{m}^{1}) =u^{1}, \gamma_{m}(u_{m}^{2}) =u^{2}$.

We begin with a  technical result, which is probably known, but 
we provide   a proof for completeness. Let us observe that as 
$\sigma(G)$ contains $\Z^{2}$, $G$ is acyclic (i.e. is not cyclic).

\begin{prop}
\label{proof:10}
Let $G$ be an acyclic, simple dimension group, let $x_{1}, \ldots, x_{J}$
be strictly positive elements of $G$ and let $n \geq 1$. 
There exist $I \geq 1$, strictly positive elements
$y_{1}, \ldots, y_{I}$  of $G$ and integers 
$ A(i,j), 1 \leq i \leq I, 1 \leq j \leq J$ such that $A(i,j) \geq n$, for
all $1 \leq i \leq I, 1 \leq j \leq J$, and 
\[
x_{j} = \sum_{i=1}^{I} A(i,j)y_{i},
\]
for all $1 \leq j \leq J$.
\end{prop}

\begin{proof}
By Lemma 14.4 of Goodearl \cite{Good:book}, we may find $x$ in $G$ with
 $ 0 < x \leq x_{j}$, for all $1 \leq j \leq J$.
 Then by Lemma 14.5 of \cite{Good:book}, we may find $y > 0$ such that $ny \leq x$.
 We may then find $\Z^{I_{m}}$ in our inductive system so that 
 $y,x, x_{1}, \ldots, x_{J}$ are all represented by elements
 of $\Z^{I_{m}}$ and all the inequalities above hold there. Additionally, we
 may assume the representative for $y$ is an order unit in $\Z^{I_{m}}$.
 The conclusion follows by letting $I=I_{m}$ and 
  $y_{i} = \gamma_{m}(\varepsilon_{i}), 1 \leq i \leq I$.
  \end{proof}
  
We next provide an improvement on this result.

\begin{prop}
\label{proof:20}
Let $(G, G^{+})$ be a simple, acyclic dimension group.
Suppose that $\eta: (\Z^{J}, \Z^{J+}) \rightarrow (G, G^{+})$
is a positive group homomorphism, $v$ is an order unit in
$\Z^{J+}$ and $s > 1$ is rational. There exists a 
positive integer $K_{0}$ such that, for all
$K \geq K_{0}$, there is a $2J \times J$ non-negative matrix $B$
and 
a positive group homomorphism
 $\eta': (\Z^{2J}, \Z^{2J+}) \rightarrow (G, G^{+})$
such that
\begin{enumerate}
\item $\eta' \circ B = \eta$,
\item $K \leq (Bv)(j) \leq sK$, for all $ 1 \leq j \leq 2J$.
\end{enumerate}
\end{prop}

\begin{proof}
We choose 
\[
K_{0} > (s-1)^{-1} 2 \prod_{j=1}^{J} v(j) > 0.
\]
If $K \geq K_{0}$, then $sK - K > 2 \prod_{j=1}^{J} v(j)$ and it 
follows that there
are at least two consecutive integer multiples of
 $\prod_{j=1}^{J} v(j)$ between 
$K$ and $sK$. Call the first $N$. That is, we have 
$K \leq N \leq N+\prod_{j=1}^{J} v(j) \leq sK$
and $N$ is divisible by 
$\prod_{j=1}^{J} v(j)$.

Applying Proposition \ref{proof:10}
to $x_{j} = \eta(\varepsilon_{j}), 1 \leq j \leq J,$
and $n = N^{2} +1 $, we may find $I, y_{i}, 1 \leq i \leq I,$
and $A(i,j)$, $1 \leq j \leq J, 1 \leq i \leq I,$
as in the conclusion.

That is, for any $i,j$, we have 
$A(i,j) \geq N^{2} > \left( N v(j)^{-1} \right)^{2}$ and we may 
write
\[
A(i,j) = Q(i,j) \left( N v(j)^{-1} \right)^{2} + R(i,j),
\]
 with $Q(i,j)$ a positive integer,
  $0 < R(i,j) \leq  \left( N v(j)^{-1} \right)^{2}$.
Hence, we have 
\[
A(i,j) = \left( Q(i,j) N v(j)^{-1}  -R(i,j) \right) 
N v(j)^{-1} + R(i,j) \left( N v(j)^{-1} +1 \right).
\]

Consider the $I \times 2J$ matrix
\[
A' = \left[ (Q(i,j) N v(j)^{-1}  -R(i,j)), \hspace{1cm} R(i,j)  \right].
\]
Let $B_{1}$ be the diagonal $ J \times J$-matrix
with diagonal entries $N v(j)^{-1}, 1 \leq j \leq J$, and 
 let $B_{2}$ be the diagonal matrix
with diagonal entries $N v(j)^{-1} +1, 1 \leq j \leq J$.
 Finally, define
\[
B =\left[ \begin{matrix} B_{1} \\ B_{2} \end{matrix} \right].
\]

It follows from the equations above 
 that $A'B = A$. Letting $\eta'(\varepsilon_{j}) = \sum_{i} A'(i,j) y_{i}$,
 for each $ 1 \leq j \leq 2J$, we have 
 $\eta' \circ B = \eta$.

It is also a simple matter to see that 
\[
Bv = (N, N, \ldots, N, N+v(1), N + v(2), \ldots, N + v(J) )^{T}
\]
and, for all $j$, 
\[
K \leq N \leq N + v(j) \leq sK,
\]
so the last part of the conclusion holds.
\end{proof}

 For each integer $l \geq 2$, the vectors $(l,1), (1,l)$
 form a basis for $\R^{2}$ and it will be convenient 
 to have a notation for the dual basis (written as column vectors):
\[
h_{l}^{1} = (l^{2}-1)^{-1} \left( l, -1 \right)^{T},
 \hspace{1cm} h_{l}^{2} = (l^{2}-1)^{-1} \left(  -1, l \right)^{T}.
 \]
 We note that, if $\sigma, u^{1}, u^{2}$ are as in Theorem \ref{intro:40} so
 that $\sigma(u^{1})=(1,0)$ and  $\sigma(u^{2})=(0,1)$, then 
 $(l,1) = \sigma(l u^{1} + u^{2}), (1,l) = \sigma( u^{1} + lu^{2})$
 and also that right multiplication by 
 the matrices $h_{l}^{1}(l,1)$ and $h_{l}^{2}(1,l)$
 are the projections onto the lines of slope $l^{-1}$ and $l$ along the 
 other. We will shortly be applying these to the elements
 $\sigma \circ \gamma(\varepsilon_{i})$, calling the results
 $p^{1}(i)$ and $p^{2}(i)$.
 
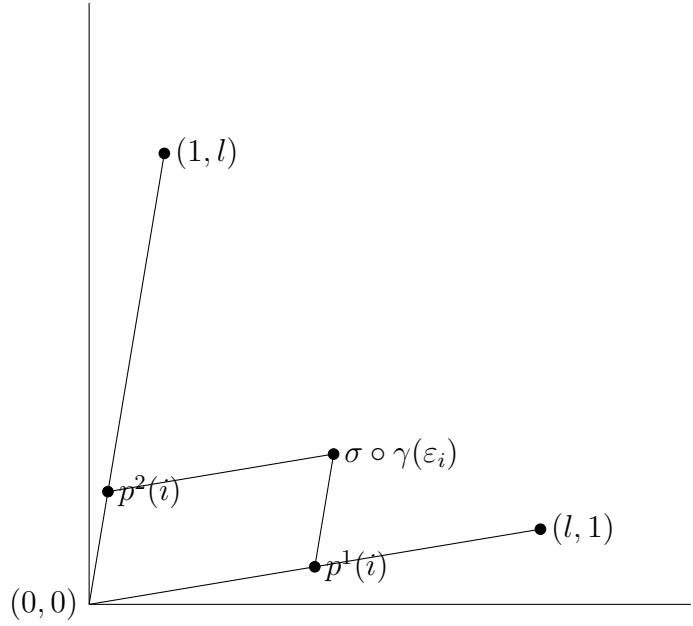
      \begin{tikzpicture}
       \draw (2,2) node[anchor=east] {$(0,0)$};   
      \filldraw (8,3)  node[anchor=west] {$(l,1)$} circle (2pt);
              \filldraw (3,8)  node[anchor=west] {$(1,l)$} circle (2pt);
  \draw (2,2) --  (10,2);
  \draw (2,2) -- (2,10);
   \draw (2,2) --  (8,3); 
   \draw (2,2) --  (3,8);
       \filldraw (5,2.5)  node[anchor=west] {$p^{1}(i)$} circle (2pt);
         \filldraw (2.25,3.5)  node[anchor=west] {$p^{2}(i)$} circle (2pt); 
           \filldraw (5.25,4)  node[anchor=west]
            {$\sigma \circ \gamma(\varepsilon_{i})$} circle (2pt);  
     \draw (5,2.5) --  (5.25,4);  
         \draw (2.25,3.5) --  (5.25,4);            
    \end{tikzpicture}

  For each  integer $l \geq 2$, we define
 subsets of $\R^{2}$
 by 
 \begin{eqnarray*}
 C(l,1) & = &  \{ (x, y) \in \R^{2} \mid 0 < (l-1)y < x <(l+1)y \}, \\
 C(l,2) & = &  \{ (x, y) \in \R^{2} \mid 0 < (l-1)x < y <(l+1)x \}. 
 \end{eqnarray*}
 Observe that these are cones in the first quadrant containing the rays
 of slopes $l^{-1}$ and $l$, respectively. In particular, the former
  contains $(l,1)$ and the latter
 contains $(1,l)$.

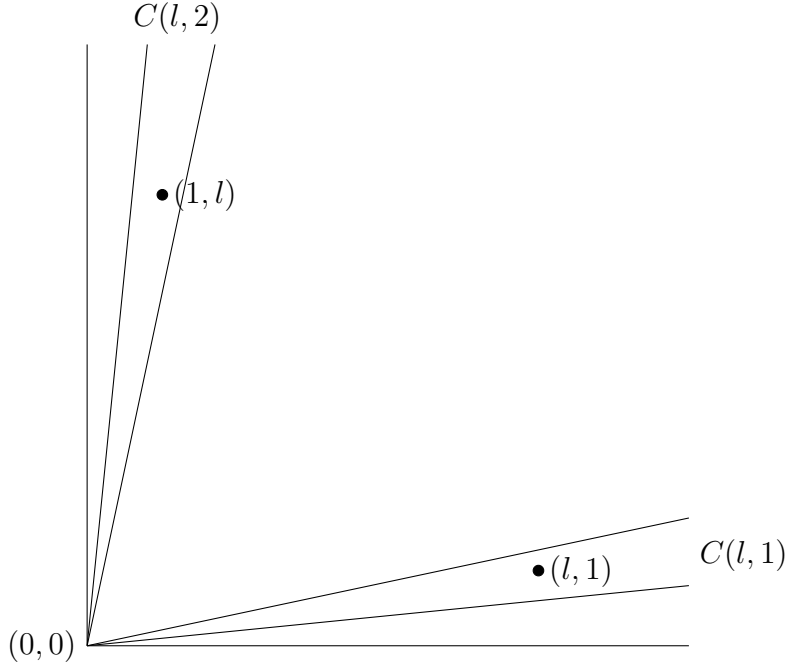
      \begin{tikzpicture}
       \draw (2,2) node[anchor=east] {$(0,0)$};   
      \filldraw (8,3)  node[anchor=west] {$(l,1)$} circle (2pt);
          \draw (10,3.2) node[anchor=west] {$C(l,1)$};
              \filldraw (3,8)  node[anchor=west] {$(1,l)$} circle (2pt);
          \draw (3.2,10) node[anchor=south] {$C(l,2)$};
  \draw (2,2) --  (10,2);
  \draw (2,2) -- (2,10);
   \draw (2,2) --  (10,3.7);
  \draw (2,2) -- (10,2.8); 
   \draw (2,2) --  (3.7,10);
  \draw (2,2) -- (2.8,10);   
    \end{tikzpicture}

\vspace{1cm}
We also define $G(l,i) = \{ (0,0) \} \cup \sigma^{-1}(C(l,i))$, 
for $i=1,2$. That is, we have 
\begin{eqnarray*}
 G(l,1) & = &  \{ g \in G \mid 0 < (l-1)\sigma_{2}(g) < \sigma_{1}(g) <(l+1)
 \sigma_{2}(g)  \}, \\
 G(l,2) & = &  \{ g\in G \mid 0 < (l-1)\sigma_{1}(g) < \sigma_{2}(g) <(l+1)
 \sigma_{1}(g) \}.
 \end{eqnarray*}

The same argument as in the proof of Theorem \ref{intro:10} shows that, for
 any $l \geq 2, i=1,2$,
$(G, G(l,i))$ is a simple acyclic dimension group. 
The identity map on $G$ is a positive group homomorphism from
$(G, G(l,i))$ to $(G, G^{+})$. Moreover, if 
$l \geq 2$ is a positive integer, $lu^{1} + u^{2}$ is an order unit
in $(G, G(l,1))$ while $u^{1} + lu^{2}$ is an order unit 
for $(G, G(l,2))$.

We will begin with one of the groups in our inductive system,
 $(\Z^{I}, \Z^{I+})$, along with its canonical map, $\gamma$, into $(G, G^{+})$.
 Our first task will be to factor $\gamma = \gamma^{1} + \gamma^{2}$, where
 $\gamma^{i}:(\Z^{I}, \Z^{I+}) \rightarrow (G, G(l,i))$, for $i=1,2$.
 This will involve some careful estimates. In particular, we will 
 choose a positive integer $m$, let $s = 1 + m^{-1}$ and consider
 a number of estimates of the form $s^{-1}A \leq B \leq sA$.  Our choice
 for $l$ will depend on the parameter $m$.

  \begin{prop}
 \label{proof:40}
 Let $I \geq 1$, $l_{0} \geq 2$ and $ m \geq 2$   be  integers.
 Set  $s = 1 + m^{-1}$. Suppose
 $\gamma: (\Z^{I}, \Z^{I+}) \rightarrow (G, G^{+})$ is a strictly
 positive homomorphism and $v^{1}, v^{2}$ in $\Z^{I}$ are
 such that $\gamma(v^{i}) = u^{i}, i = 1,2$.  Then
  there exists an integer  \newline
  $l \geq l_{0}$, $s^{1}, s^{2}$ in $\Z^{I+}$
   and
   positive homomorphisms \newline
  $\gamma^{1}: (\Z^{I}, \Z^{I+}) \rightarrow (G, G(l,1))$
  and $\gamma^{2}: (\Z^{I}, \Z^{I+}) \rightarrow (G, G(l,2))$
  such that 
  \begin{enumerate}
  \item 
  $\gamma^{1} + \gamma^{2} = \gamma$. 
 \item 
   \[
  s^{-2} \sigma_{j} \circ \gamma(\varepsilon_{i}) < ls^{j}(i)^{-1} 
  <   s \sigma_{j} \circ \gamma(\varepsilon_{i})
   \]
   for $ 1 \leq i \leq I$.
 \item 
Set $w^{1} = lv^{1} + v^{2}$ and $w^{2} = v^{1} + lv^{2}$.  Letting 
  \begin{eqnarray*}
F^{1}  & =   & \{w^{1} , 
 s^{1}(i) \varepsilon_{i} - m w^{1},
  (m+1) w^{1} - s^{1}(i) \varepsilon_{i}, \\
  &  &      w^{1}\pm ml w^{2}   \mid 1 \leq i \leq I \}, \\
 F^{2} & = &  \{w^{2} , 
 s^{2}(i) \varepsilon_{i} - m w^{2},
  (m+1) w^{2} - s^{2}(i) \varepsilon_{i}, \\
  &  &      w^{2}\pm ml w^{1}   \mid 1 \leq i \leq I \},
\end{eqnarray*}
we have $\gamma^{j}(F^{j}) \subseteq G(l,j)$, for $j=1,2$.
\item 
\begin{eqnarray*}
(m+1) \gamma^{1}(w^{1}) - m (lu^{1} + u^{2}) & \in & G(l,1), \\
(m+1)  (lu^{1} + u^{2}) -m  \gamma^{1}(w^{1}) & \in & G(l,1), \\
(m+1) \gamma^{2}(w^{2}) - m (u^{1} + lu^{2}) & \in & G(l,2), \\
(m+1)  (u^{1} + lu^{2}) -m  \gamma^{2}(w^{2}) & \in & G(l,2).
\end{eqnarray*}
\end{enumerate}
  \end{prop}

 \begin{proof}
For $1 \leq i \leq I$, 
let $\sigma\circ \gamma(\varepsilon_{i})= (x_{i}, y_{i})$. 
By definition of $\sigma_{1}, \sigma_{2}$, 
we  have $\sigma_{1} \circ \gamma( \varepsilon_{i}) = x_{i}$
and $\sigma_{2} \circ \gamma( \varepsilon_{i}) = y_{i}$.
As $\gamma$ is strictly positive, we have that 
$x_{i}, y_{i} > 0$.

Note that for $l > \max\{ x_{i}^{-1}y_{i}, x_{i}y_{i}^{-1}, 2 \}$, we have 
\begin{eqnarray}
  \label{eqn:04}
\sigma \circ \gamma( \varepsilon_{i})h_{l}^{1} & =  & 
\frac{1}{l^{2}-1} (lx_{i} - y_{i})  > 0, \\
 \label{eqn:05}
\sigma \circ \gamma( \varepsilon_{i})h_{l}^{2} & =  & 
\frac{1}{l^{2}-1} (-x_{i} +l y_{i}) > 0. 
\end{eqnarray}

Since for $j=1,2$, 
\begin{eqnarray*}
\lim_{l \rightarrow \infty} l \sigma\circ \gamma(\varepsilon_{i})h^{j}_{l}
   & = & \sigma_{j} \circ \gamma(\varepsilon_{i}) \\
  \lim_{l \rightarrow \infty}  \sigma\circ \gamma(\varepsilon_{i})
  h^{j}_{l}
   & = &  0,
   \end{eqnarray*}
   there exists $l \geq l_{0}$ such that for all $ 1 \leq i \leq I$ and 
   $j=1,2$, we have  
  \begin{eqnarray} 
  \label{eqn:104}
  s^{-1} \sigma_{j} \circ \gamma(\varepsilon_{i}) & < l \sigma \circ \gamma(\varepsilon_{i}) h_{l}^{j}
   < & s\sigma_{j} \circ \gamma(\varepsilon_{i}), \\
   \label{eqn:105}
  0 & <  \sigma \circ \gamma(\varepsilon_{i}) h_{l}^{j}
   < & m^{-1}.
   \end{eqnarray}

 By (\ref{eqn:105}), for all $ 1 \leq i \leq I$ and $j=1,2$, 
    we may find positive integers 
   $s^{j}(i)$ such that
   \begin{eqnarray}
   \label{eqn:06}
   1 <  s^{j}(i) \sigma \circ \gamma(\varepsilon_{i}) h_{l}^{j}
   < s.
   \end{eqnarray}
The estimates given for $s^{j}(i)$ in 
part 2 of the conclusion 
are easy consequences of the inequalities of (\ref{eqn:104}) 
 and (\ref{eqn:06}).

  Recall that 
  \begin{eqnarray*}
(l, 1) h_{l}^{1} & = (1,l) h_{l}^{2} & = 1, \\
(l,1) & = \sigma(lu^{1} + u^{2}) & = \sigma \circ \gamma(lv^{1} + v^{2}),
 \\
 (1,l) & = \sigma(u^{1} + lu^{2}) & = \sigma \circ \gamma(v^{1} + lv^{2}). 
  \end{eqnarray*}
  Then we have 
  \begin{eqnarray}
  (l,1) & = & (l,1) h^{1}_{l} (l,1) \\
  (l,1)    & = & \sigma \circ \gamma(lv^{1} + v^{2}) h_{l}^{1}(l,1) \\
   (l,1)     &  =  &  \sum_{i'=1}^{I} (lv^{1}(i') + v^{2}(i'))   \label{eqn:07}
        \sigma \circ \gamma(\varepsilon_{i'})
      h_{l}^{1}(l,1)
\end{eqnarray} 
 and, in a similar way,
\begin{eqnarray} \label{eqn:08}
(1,l) = \sum_{i'=1}^{I} (v^{1}(i') + lv^{2}(i')) 
        \sigma \circ \gamma(\varepsilon_{i'})
      h_{l}^{2}(l,1).
\end{eqnarray}
For the rest of the 
proof, all sums will be over $1 \leq i' \leq I$.

As $(1,l)h_{l}^{1} = (l,1)h^{2}_{l}=0$, we get in a similar way
\begin{eqnarray}
 0 & = & (1,l)h_{l}^{1} \\
    &  =  &  \sigma \circ \gamma(v^{1} + lv^{2})h_{l}^{1} \\ \label{eqn:09}
    &  =  &  \sum (v^{1}(i') + lv^{2}(i'))  
        \sigma \circ \gamma(\varepsilon_{i'})
      h_{l}^{1}
      \end{eqnarray}
      and 
\begin{eqnarray}
\label{eqn:10}
  0 =     \sum (lv^{1}(i') + v^{2}(i'))  
        \sigma \circ \gamma(\varepsilon_{i'})
      h_{l}^{2}.
      \end{eqnarray}
  
  Note that if $p^{1}$ ($p^{2}$, respectively) denotes the 
  projection of $\R^{2}$ onto $\R(l,1)$ along $\R(1,l)$ 
  (onto $\R(1,l)$ along $\R(l,1)$, respectively), we have 
  \[
  p^{1}(x,y) = (x,y)h_{l}^{1}(l,1), \hspace{1cm}
   p^{2}(x,y) = (x,y)h_{l}^{2}(1,l). 
  \] 
  
To simplify our notation, for $1 \leq i \leq I$, we set
\begin{eqnarray*}
p^{1}(i) & = p^{1}(\sigma \circ \gamma(\varepsilon_{i})) & = 
\sigma \circ \gamma(\varepsilon_{i})h_{l}^{1}(l,1), \\
p^{2}(i) & = p^{2}(\sigma \circ \gamma(\varepsilon_{i})) & 
= \sigma \circ \gamma(\varepsilon_{i})h_{l}^{2}(1,l).
\end{eqnarray*}
By (\ref{eqn:05}), for $1 \leq i \leq I$ and $j=1,2$, 
$p^{j}(x,y)$ lies in $C(l,j)$. 
      
 Now we make the simple observation that for any real numbers $ r > t$, 
 \[
 r(l,1) - t(l,1) \in C(l,1), \hspace{1cm} 
 r(1,l) - t(1,l) \in C(l,2).
 \]

Then it follows from  (\ref{eqn:06}), (\ref{eqn:07})  and    equation
 (\ref{eqn:10}) 
that, for  $1 \leq i \leq I$, we have
\begin{eqnarray*}
  & s^{1}(i)p^{1}(i) -  \sum (lv^{1}(i') + v^{2}(i')) p^{1}(i') & \\
  & s  \sum (lv^{1}(i') + v^{2}(i')) p^{1}(i')  -  s^{1}(i)p^{1}(i) & \\
  &  \sum \left( lv^{1}(i') + v^{2}(i') \pm ml 
(v^{1}(i') + lv^{2}(i'))  \right)p^{1}(i')  &
  \end{eqnarray*}
are in $C(l,1)$.

 Finally,  from (\ref{eqn:07}), it follows easily 
 that
\[
s (l, 1) - \sum (lv^{1}(i') + v^{2}(i')) p^{1}(i') 
\]
 is in $C(l,1)$.
 
 Let us summarize by listing these elements
 \begin{eqnarray*}
 C(l,1) & \supseteq & \{
p^{1}(i),  \\
  & &   s^{1}(i)p^{1}(i) -\sum (lv^{1}(i') + v^{2}(i')) p^{1}(i'),  \\
  & & s\sum (lv^{1}(i') + v^{2}(i')) p^{1}(i') -s^{1}(i)p^{1}(i), \\
  & &   \sum \left( lv^{1}(i') + v^{2}(i') \pm ml
   (v^{1}(i') + l v^{2}(i') ) \right) p^{1}(i'), \\
  & & s (l, 1) - \sum (lv^{1}(i') + v^{2}(i')) p^{1}(i'), \\
 & & s \sum (lv^{1}(i') + v^{2}(i')) p^{1}(i') - (l,1) \\
   &  &  \mid 1 \leq i \leq I \}. 
 \end{eqnarray*}

There is a similar list of elements of $C(l,2)$ involving $p{^2}$
rather than $p^{1}$, which we will not list here. However,
 we know that
\begin{eqnarray*}
p^{1}(i) + p^{2}(i) & = & 
\sigma \circ \gamma(\varepsilon_{i})h_{l}^{1}(l,1)
+ \sigma \circ \gamma(\varepsilon_{i})h_{l}^{2}(1,l)\\
  &  =  &  \sigma \circ \gamma(\varepsilon_{i}).
  \end{eqnarray*}
  Replacing each occurrence of $p^{2}(i)$ by 
  $\sigma \circ \gamma(\varepsilon_{i}) - p^{1}(i)$, 
  we obtain the following list of elements of $C(l,2)$:

 \begin{eqnarray*}
 C(l,2) & \supseteq & \{
(\sigma \circ \gamma(\varepsilon_{i}) - p^{1}(i)),  \\
  & &   s^{2}(i) (\sigma \circ \gamma(\varepsilon_{i}) - p^{1}(i))
    \\
      &  &  -\sum (v^{1}(i') + lv^{2}(i')) 
   (\sigma \circ \gamma(\varepsilon_{i'}) - p^{1}(i')),  \\
  & & s\sum (v^{1}(i') + lv^{2}(i')) 
  (\sigma \circ \gamma(\varepsilon_{i'}) - p^{1}(i')) \\
    &  &  -s^{2}(i) (\sigma \circ \gamma(\varepsilon_{i}) - p^{1}(i)), \\
   & &   \sum \left( v^{1}(i') + lv^{2}(i') \pm ml
   (lv^{1}(i') +  v^{2}(i') ) \right) 
   (\sigma \circ \gamma(\varepsilon_{i'}) - p^{1}(i')), \\
  & & s ( 1, l ) - \sum (v^{1}(i') + lv^{2}(i')) (\sigma \circ \gamma(\varepsilon_{i'}) - p^{1}(i')), \\
 & & s \sum (v^{1}(i') + lv^{2}(i')) (\sigma \circ \gamma(\varepsilon_{i'}) - p^{1}(i')) - (1, l) \\
   &  &  \mid 1 \leq i \leq I \}. 
 \end{eqnarray*}

It follows from our hypotheses and Theorem \ref{intro:20}
 that $\sigma(G)$ is dense
 in $\R^{2}$ and the facts that both $C(l,1)$ and 
$C(l,2)$ are  open to conclude that, for each $1 \leq i \leq I$, we may find
$g_{i}$ in $G$ such that $\sigma(g_{i})$ is sufficiently close
to $p^{1}(i)$ so that if we replace $p^{1}(i)$ by $\sigma(g_{i})$, 
the two containments above remain valid.

We  then define the two homomorphisms $\gamma^{1}$ and 
$\gamma^{2}$ from $\Z^{I}$ to $G$ by setting
$\gamma^{1}(\varepsilon_{i}) = g_{i}$ and 
 $\gamma^{2}(\varepsilon_{i}) = \gamma(\varepsilon_{i}) - g_{i}$, 
 for each $ 1 \leq i \leq I$. 
 The first part of the conclusion  of the proposition
 is immediate.
 The positivity of
 $\gamma^{1}$ and $\gamma^{2}$ is a consequence of the choice
 of $g_{i}$ and the inclusion of $p^{1}(i)$ and 
 $\sigma \circ \gamma(\varepsilon_{i}) - p^{1}(i)$ in 
 the sets $C(l,1)$ and $C(l,2)$, respectively. 
 
 The facts that, for all $1 \leq i \leq I$,
 \begin{eqnarray*}
 s^{1}(i) \sigma \circ \gamma^{1}(\varepsilon_{i}) - 
 \sum (lv^{1}(i') + v^{2}(i')) \sigma \circ \gamma^{1}(\varepsilon_{i'})
 & \in & C(l,1), \\ 
  s\sum (lv^{1}(i') + v^{2}(i')) 
    \sigma \circ \gamma^{1}(\varepsilon_{i'})  - s^{1}(i) 
   \sigma \circ \gamma^{1}(\varepsilon_{i})& \in & C(l,1) 
 \end{eqnarray*}
 imply that 
  \begin{eqnarray*}
 s^{1}(i) \gamma^{1}(\varepsilon_{i}) - 
  \gamma^{1}(lv^{1} + v^{2})
 & \in & G(l,1), \\ 
    (m+1)  \gamma^{1}(lv^{1} + v^{2} )  - m s^{1}(i) 
     \gamma^{1}(\varepsilon_{i})& \in & G(l,1). 
 \end{eqnarray*}
 
 Further, the elements
\[
\sum \left( lv^{1}(i') + v^{2}(i') \pm ml
   (v^{1}(i') +  lv^{2}(i') ) \right) 
   \sigma \circ
    \gamma^{1}(\varepsilon_{i'}) \in C(l,1)
\] 
 imply that 
$\gamma^{1}((lv^{1}+v^{2})  \pm m (v^{1} + lv^{2}))$
 are in $G(l,1)$. This completes the proof that
  $\gamma^{1}(F^{1}) \subseteq G(l,1)$.
 The inclusion $\gamma^{2}(F^{2}) \subseteq G(l,2)$ 
 is done in a similar way.
 
 The last set of assertions of the proposition
  follows from the containments in $C(l,1)$
  and $C(l,2)$ we have and  the facts that 
 $\sigma(lu^{1} + u^{2})=(l,1)$, \newline 
 $\sigma(u^{1} + lu^{2})=(1,l)$.
\end{proof}

\begin{prop}
\label{proof:50}
 Let $I \geq 1,  m \geq 2,  a \geq 1$
  and $K_{0} \geq 1$
   be   integers and let $l_{0} > m$.
 Suppose
 $\gamma: (\Z^{I}, \Z^{I+}) \rightarrow (G, G^{+})$ is a 
 strictly
 positive homomorphism and $v^{1}, v^{2}$ in $\Z^{I}$ are
 such that $\gamma(v^{j}) = u^{j}, j = 1,2$.
 Also, let $l \geq l_{0},
  \gamma^{1}, \gamma^{2}, s^{1}, s^{2}, F^{1}$ and $ F^{2}$ 
 be as in  Proposition \ref{proof:40}.
 Then there exist  positive integers 
 $J, K \geq K_{0}$, $J \times I$-matrices $B^{1}, B^{2}$
 with positive integer entries
 and positive homomorphisms 
 $\eta^{1}: (\Z^{J}, \Z^{J+}) \rightarrow (G, G(l,1)),
 \eta^{2}: (\Z^{J}, \Z^{J+}) \rightarrow (G, G(l,2))$
 such that, using $s = 1 + m^{-1}$,
 \begin{enumerate}
 \item $\eta^{1} \circ B^{1} = 
 \gamma^{1}, \eta^{2} \circ B^{2} = \gamma^{2}$,
 \item 
   \[
l K  \leq (B^{1}(lv^{1}+v^{2}))(j), (B^{2}(v^{1}+lv^{2}))(j) \leq  slK, 
 \]
 for all $ 1 \leq j \leq J$.
 \item 
  \begin{eqnarray*}
  a  & \leq 
  s^{-2} m  K \sigma_{1} \circ \gamma(\varepsilon_{i})   
 \leq  &  B^{1}(j,i),  \\ 
   a  & \leq 
  s^{-2} m  K \sigma_{2} \circ \gamma(\varepsilon_{i})   
  \leq  &  B^{2}(j,i), 
  \end{eqnarray*}
  and 
  \begin{eqnarray*} 
 B^{1}(j,i) &   \leq  &
 s^{2} (m+1)  K \sigma_{1} \circ \gamma(\varepsilon_{i})
 = s^{3}m K \sigma_{1} \circ \gamma(\varepsilon_{i}),  \\
 B^{2}(j,i)  & \leq  &
 s^{2} (m+1)  K \sigma_{2} \circ \gamma(\varepsilon_{i})
 = s^{3} m  K \sigma_{2} \circ \gamma(\varepsilon_{i}), 
     \end{eqnarray*}
for all $1 \leq i \leq I, 1 \leq j \leq J$, 
     and 
     \item 
  \begin{eqnarray*}
    K \leq & 
  (B^{1}v^{1})(j),  (B^{2}v^{2})(j) &  \leq  s^{2}  K, \\
   s^{-2}K \leq & 
  -l(B^{1}v^{2})(j),  -l(B^{2}v^{1})(j) &  \leq  s^{3}  K,
\end{eqnarray*}
for all $1 \leq j \leq J$.
\end{enumerate}
\end{prop}

\begin{proof}
We first observe that, by
 hypotheses, $\gamma(\varepsilon_{i})$ is strictly positive, for all
$1 \leq i \leq I$, and it follows that, for $j=1,2$, 
$\sigma_{j} \circ \gamma(\varepsilon_{i}) $ is a  strictly positive
real number.

As $(G, G(l,1))$ and $(G, G(l,2))$ are simple acyclic dimension groups, 
the maps $\gamma^{1}, \gamma^{2}$ can be factored as in part 1 
of the conclusion
in such a way that $B^{1}(F^{1}), B^{2}(F^{2})$ consist
 of positive vectors, at least for 
possibly different values of $J$. However, the 
value of $J$ may be increased
as follows. Since $(G, G(l,1))$ is a simple acyclic 
dimension group, we 
may find $ 0 < g < \eta^{1}(\varepsilon_{J})$. 
We replace the matrix $B^{1}$ by a $J+1 \times I$ matrix whose first
$J$ rows are the same as those in $B^{1}$ and whose row $J+1$  is the same
as row $J$ of $B^{1}$. In other words, we duplicate the last row.
We also replace the function $\eta^{1}$ by the function which
sends $\varepsilon_{j}$ to $ \eta^{1}(\varepsilon_{j})$, for
all $j < J$, sends $\varepsilon_{J}$ to $g$ and sends
 $\varepsilon_{J+1}$ to $\eta^{1}(\varepsilon_{J}) -g$. It is an 
 easy matter to check this new matrix and function satisfy 
 the desired conclusion.
From this, we see that the same $J$ may be 
used for both $\eta^{1}, B^{1}$ and 
$\eta^{2}, B^{2}$. 

Applying Proposition \ref{proof:20} to 
$\eta^{1}, v = B^{1}(lv^{1} + v^{2})$ and $s=  1 + m^{-1}$
 produces $\eta'$  and $B$. We will now let $\eta^{1}$ be $\eta'$ and 
 let $B^{1}$ be $B \cdot B^{1}$. This also changes $J$ to $2J$, but we will continue to use $J$.
 We also apply Proposition \ref{proof:20} to 
 $\eta^{2}, v = B^{2}(v^{1} + lv^{2})$ and $s=  1 + m^{-1}$. 
The  $K$ given by  Proposition \ref{proof:20}
may be chosen arbitrarily
 large in each case, so we will choose the same $K$ for both. In addition, 
 we will choose it to be a multiple of $l$; that is, we will have  $lK$ satisfying the conclusion for some $K$. Moreover, we may
 choose this so that $K \geq K_{0}$ in our hypothesis and
\[
K \sigma_{j} \circ \gamma(\varepsilon_{i})^{-1} \geq \frac{9a}{8}
\]
for all $ j=1,2, 1 \leq i \leq I$. 

  Then the first inequality of
 the third condition will follow from the facts that
  $K \geq K_{0}$ and
 $\frac{8}{9} \leq m s^{-2}$.

 We can now assume that 
 \[
 lK  \leq (B^{1}(lv^{1}+v^{2}))(j), (B^{2}(v^{1}+lv^{2}))(j) \leq  slK, 
 \]
 for all $ 1 \leq j \leq J$.
 Henceforth, we focus on $\eta^{1}, B^{1}$ as the other case is similar.
 
 For $1 \leq i \leq I$, we know 
 $(m+1)(lv^{1}+v^{2}) - s^{1}(i)\varepsilon_{i}$ 
 is in $F^{1}$, so,
 for $ 1 \leq j \leq J$, we have
 \begin{eqnarray*}
 B^{1}(j,i) & = & (B^{1}\varepsilon_{i})(j) \\
     &  \leq   &  s^{1}(i)^{-1} (m+1) B^{1}(lv^{1}+v^{2})(j) \\
     & \leq & s^{1}(i)^{-1} (m+1) s lK \\
     &  \leq & s^{2}(m+1)  K \sigma_{1} \circ \gamma(\varepsilon_{i}) \\
     & = & s^{3}mK\sigma_{1} \circ \gamma(\varepsilon_{i}) 
     \end{eqnarray*}
     using part 2 of Proposition \ref{proof:40}   and the fact that
     $m+1 = m(1 + m^{-1}) = ms$.
     Similarly, we have 
 \begin{eqnarray*}
 B^{1}(j,i) & = &  (B^{1}\varepsilon_{i})(j) \\
     &  \geq   & s^{1}(i)^{-1} m B^{1}(lv^{1}+v^{2})(j) \\
   & \geq & s^{1}(i)^{-1} m l K \\
    &  \geq & s^{-2} m  K \sigma_{1} \circ \gamma(\varepsilon_{i}).
   \end{eqnarray*} 
   With  similar computations for $B^{2}$, the
    third part of the conclusion holds.

   As $B^{1}(F^{1})$ consists of positive vectors, we have 
   \begin{eqnarray*}
   -B^{1}(l v^{1} + v^{2}) & \leq   ml B^{1}( v^{1} + lv^{2})
  \leq & B^{1}(l v^{1} + v^{2}). 
  \end{eqnarray*}
  Using the fact that $B^{1}(l v^{1} + v^{2})$ is positive, we have 
  \begin{eqnarray*}
  B^{1}v^{1} & = & (l^{2}-1)^{-1} \left( l B^{1}(l v^{1} + v^{2}) 
   - B^{1}( v^{1} + l v^{2}) \right) \\
   & \leq & (l^{2}-1)^{-1} \left( l B^{1}(l v^{1} + v^{2}) 
     + m^{-1}l^{-1}  B^{1}(l v^{1} + v^{2}) \right) \\
     & = & (l^{2}-1)^{-1}  (l + m^{-1}l^{-1}) B^{1}(l v^{1} + v^{2}) \\
      &  \leq  &  (1 - l^{-2})^{-1} (1 + m^{-1}l^{-2})  s K \\
       & \leq & s^{2}K,
       \end{eqnarray*}
  as $(1-l^{-2})^{-1}(1 + m^{-1}l^{-2})\leq s$ 
  for $l > m$.
  
  We also have 
   \begin{eqnarray*}
  B^{1}v^{1} & = & (l^{2}-1)^{-1} \left( l B^{1}(l v^{1} + v^{2}) 
   - B^{1}( v^{1} + l v^{2}) \right) \\
   & \geq & (l^{2}-1)^{-1} \left( l B^{1}(l v^{1} + v^{2}) 
     - m^{-1}l^{-1}  B^{1}(l v^{1} + v^{2}) \right) \\
     & = & (l^{2}-1)^{-1}  (l - m^{-1}l^{-1}) B^{1}(l v^{1} + v^{2}) \\
      &  \geq  &  (1 - l^{-2})^{-1} (1 - m^{-1}l^{-2})   K \\
       & \geq & K.
       \end{eqnarray*}
       
       In a similar way, we compute
  \begin{eqnarray*}
  B^{1}v^{2} & = & (l^{2}-1)^{-1} \left( l B^{1}( v^{1} + l v^{2})
  -  B^{1}(l v^{1} + v^{2}) 
    \right) \\ 
      &  \leq  &  (l^{2}-1)^{-1} \left(  m^{-1} B^{1}(l v^{1} + v^{2})   
       -  B^{1}(l v^{1} + v^{2}) 
    \right) \\
      &  =  &  -(l^{2}-1)^{-1}(1 - m^{-1}) B^{1}(l v^{1} + v^{2}) \\
      &  \leq   &  - (1 - l^{-2})^{-1} (1 - m^{-1}) s l^{-1} K \\
      & \leq & -s^{-2} l^{-1}K 
      \end{eqnarray*}
      since $s^{2}(1 - m^{-1}) \geq 1$, for $m \geq 2$.

Finally, we have 
 \begin{eqnarray*}
  B^{1}v^{2} & = & (l^{2}-1)^{-1} \left( l B^{1}( v^{1} + l v^{2})
  -  B^{1}(l v^{1} + v^{2}) 
    \right) \\ 
      &  \geq  &  (l^{2}-1)^{-1} \left( - m^{-1} B^{1}(l v^{1} + v^{2})   
       -  B^{1}(l v^{1} + v^{2}) 
    \right) \\
      &  =  &  -(l^{2}-1)^{-1}(1 + m^{-1}) B^{1}(l v^{1} + v^{2}) \\
      &  \geq   &  - (1 - l^{-2})^{-1} s^{2} l^{-1}K \\
        &  \geq   &  -  s^{3} l^{-1}K
      \end{eqnarray*}
     as $( 1- l^{-2})^{-1} \leq s$ for $l > m$.
\end{proof}

We now return to the inductive system given at the start of this section
 for
the group $(G, G^{+})$. We will choose $m \geq 2$ and $l_{0} > m$  and consider $n \geq 1$, $I = I_{n}$, 
$v^{1} = u_{n}^{1}, v^{2} = u_{n}^{2}$
and $\gamma = \gamma_{n}$. 
We first apply Proposition \ref{proof:40} to obtain $l, s^{1}, s^{2},
\gamma^{1}, \gamma^{2}$. We then apply Proposition \ref{proof:50}
to obtain $J, K, B^{1}, B^{2}, \eta^{1}$ and $\eta^{2}$. All of this will be assumed
without explicit statement for the rest of the section.

\begin{prop}
\label{proof:60}
There exist $\tilde{n} > n$,  $I_{\tilde{n}} \times J$-matrices
$C^{1}, C^{2}$ with positive integer entries such that
\begin{enumerate}
\item 
\[
C^{1}B^{1}+ C^{2}B^{2} =  A_{\tilde{n}-1} \cdots A_{n}.
\]
\item For all $ 1 \leq \tilde{i} \leq I_{\tilde{n}}$, we have
\begin{eqnarray*}
s^{-1} \left( lu^{1}_{\tilde{n}}(\tilde{i}) + u^{2}_{\tilde{n}}(\tilde{i}) \right) & \leq  
 \left( C^{1}B^{1}(l v^{1} + v^{2}) \right)(\tilde{i}) \leq & 
 s \left( lu^{1}_{\tilde{n}}(\tilde{i}) + 
 u^{2}_{\tilde{n}}(\tilde{i}) \right), \\ 
 s^{-1} \left(  u^{1}_{\tilde{n}}(\tilde{i}) + 
 lu^{2}_{\tilde{n}}(\tilde{i}) \right)
 & \leq  
 \left( C^{2}B^{2}( v^{1} + lv^{2}) \right)(\tilde{i}) 
    \leq & s \left( u^{1}_{\tilde{n}}(\tilde{i})
     + lu^{2}_{\tilde{n}}(\tilde{i}) 
    \right).
\end{eqnarray*}
\end{enumerate}
\end{prop}

\begin{proof}
The positive group homomorphisms 
$\eta^{i}: (\Z^{J}, \Z^{J+}) \rightarrow (G, G^{+})$
can be factored through the inductive limit as 
$\eta^{i} = \gamma_{\tilde{n}} \circ C^{i}$, with $C^{i}$ a
positive $I_{\tilde{n}} \times J$ matrix, for $i = 1,2$.
We have  
\[
 \gamma_{\tilde{n}} \circ (C^{1}B^{1}+ C^{2}B^{2} ) = 
\eta^{1} \circ B^{1} + \eta^{2} \circ B^{2} = 
\gamma^{1} + \gamma^{2} = \gamma_{n}
= \gamma_{\tilde{n}} \circ   A_{\tilde{n}-1} \cdots A_{n}.
\]
As $\Z^{I}$ is finitely generated, the $\tilde{n}$ may be chosen 
sufficiently large so that  $C^{1}B^{1}+ C^{2}B^{2} =
 A_{\tilde{n}-1} \cdots A_{n}$  holds.

For the second, we know from part 4 of 
 Proposition \ref{proof:40} that 
 \[
 (m+1) \gamma^{1}(lv^{1} + v^{2}) - m (lu^{1} + u^{2})  
 \in  G(l,1) \subseteq G^{+}.
 \]
 We also have 
\begin{eqnarray*}
  &  (m+1) \gamma^{1}(lv^{1} + v^{2}) - m (lu^{1} + u^{2})   &  \\
  =   &
 \gamma_{\tilde{n}} \left( (m+1) C^{1}B^{1}(lv^{1} + v^{2}) - 
m (lu_{\tilde{n}}^{1} + u_{\tilde{n}}^{2}) \right) &
\end{eqnarray*}
and so we may also choose $\tilde{n}$ sufficiently large so
that the positivity holds in $\Z^{I_{\tilde{n}}}$. This establishes the
first inequality and the others are done in a similar way.
\end{proof}

Given $n$ as a starting point and, with some choices of
parameters, we have found $ \tilde{n} > n$ and a factorization
\[
A_{\tilde{n}-1}  \cdots A_{n} = \left[ \begin{array}{cc} C^{1} & C^{2} 
\end{array} \right] 
\left[ \begin{array}{c} B^{1}\\ B^{2} 
\end{array} \right].
\]
Of course, this can be iterated and by pairing the matrix products
in the other way, we produce an equivalent Bratteli diagram for $(G, G^{+})$.
For the moment, it is enough to analyze a single second iteration.

Now, we choose $\tilde{l}_{0}, \tilde{m}$  and
use Propositions \ref{proof:40} and \ref{proof:50}  again to produce
$\tilde{l}, \tilde{s}^{1}, \tilde{s}^{2}, \tilde{J}, \tilde{K},
\tilde{B}^{1}, \tilde{B}^{2}$.
Our interest now is in the matrix
\[
\left[ \begin{array}{c} \tilde{B}^{1} \\ \tilde{B}^{2} 
\end{array} \right] \left[ \begin{array}{cc} C^{1} & C^{2} 
\end{array} \right] 
 = \left[ \begin{array}{cc} \tilde{B}^{1}C^{1} &  \tilde{B}^{1}C^{2} \\
  \tilde{B}^{2}C^{1} &  \tilde{B}^{2}C^{2} \end{array} \right].
  \]
  We will not re-state these hypotheses in what follows.
  
 We have the following estimates.
 
 \begin{prop}
 \label{proof:70}
 For all $1 \leq j \leq J$, $1 \leq \tilde{j}, \tilde{j}' \leq \tilde{J}$,
 we have 
 \begin{enumerate}
 \item 
 \begin{eqnarray*}
 s^{-1} \tilde{s}^{-5}  (\tilde{B}^{1}C^{1})(\tilde{j}, j )
  & \leq l (\tilde{B}^{2}C^{1})(\tilde{j}', j )
 \leq s  \tilde{s}^{5}   (\tilde{B}^{1}C^{1})(\tilde{j}, j ), \\
 s^{-1} \tilde{s}^{-5} (\tilde{B}^{2}C^{2})(\tilde{j}, j ) 
 & \leq l (\tilde{B}^{1}C^{2})(\tilde{j}', j )
 \leq  s\tilde{s}^{5}  (\tilde{B}^{2}C^{2})(\tilde{j}, j ),
 \end{eqnarray*} 
 \item 
 \[
 s^{-1}  \tilde{s}^{-2}  \tilde{m}
  \tilde{K}    K^{-1} \leq 
  \sum_{j=1}^{J} (\tilde{B}^{1}C^{1})(\tilde{j}, j ), 
 \sum_{j=1}^{J} (\tilde{B}^{2}C^{2})(\tilde{j}, j ) \leq  s  
 \tilde{s}^{3}  \tilde{m}
  \tilde{K}    K^{-1} .
 \]
 \end{enumerate}
 \end{prop}

\begin{proof}
We consider the first part.
Recall (see the proof of Proposition \ref{proof:60}) that 
for $i=1,2$, the positive homomorphism $\eta^{i}: (\Z^{J},\Z^{J+})
\rightarrow (G, G^{+})$ factors 
through the inductive limit as
$\eta^{i} = \gamma_{\tilde{n}} \circ C^{i}$.

Let $1 \leq j \leq J$, $1 \leq \tilde{j} \leq \tilde{J}$.  Using part 
3 of Proposition \ref{proof:50},
 we have
\begin{eqnarray}
(\tilde{B}^{1}C^{1})(\tilde{j}, j ) &  = & (\tilde{B}^{1}C^{1}\varepsilon_{j})(\tilde{j})\\
     & = & \sum_{\tilde{i}=1}^{I_{\tilde{n}}} \tilde{B}^{1}(\tilde{j}, \tilde{i}) 
      (C^{1}\varepsilon_{j})(\tilde{i}) \\
     & \leq & \sum_{\tilde{i}=1}^{I_{\tilde{n}}} \tilde{s}^{3}\tilde{m}  \tilde{K}
     \sigma_{1} \circ \gamma_{\tilde{n}}(\varepsilon_{\tilde{i}})   
      (C^{1}\varepsilon_{j})(\tilde{i})  \\   \label{eq:92}
       & = &  \tilde{s}^{3}\tilde{m}  \tilde{K}
     \sigma_{1} \circ \gamma_{\tilde{n}}(C^{1}\varepsilon_{j} ) \\
         & = &  \tilde{s}^{3}\tilde{m}  \tilde{K}
     \sigma_{1} \circ \eta^{1}(\varepsilon_{j} ) \\
       &  \leq &  \tilde{s}^{3}\tilde{m}  \tilde{K}
  (l+1)   \sigma_{2} \circ \eta^{1}(\varepsilon_{j} ) 
\end{eqnarray}
and
\begin{eqnarray*}
(\tilde{B}^{1}C^{1})(\tilde{j}, j ) &  = & (\tilde{B}^{1}C^{1}\varepsilon_{j})(\tilde{j})\\
     & = & \sum_{\tilde{i}=1}^{I_{\tilde{n}}} \tilde{B}^{1}(\tilde{j}, \tilde{i}) 
      (C^{1}\varepsilon_{j})(\tilde{i}) \\
     & \geq & \sum_{\tilde{i}=1}^{I_{\tilde{n}}} \tilde{s}^{-2}
     \tilde{m}  \tilde{K}
     \sigma_{1} \circ \gamma_{\tilde{n}}(\varepsilon_{\tilde{i}})   
      (C^{1}\varepsilon_{j})(\tilde{i})  \\  
       & = &  \tilde{s}^{-2}\tilde{m}  \tilde{K}
     \sigma_{1} \circ \gamma_{\tilde{n}}(C^{1}\varepsilon_{j} ) \\
       &  = &  \tilde{s}^{-2}\tilde{m}  \tilde{K}
   \sigma_{1} \circ \eta^{1}(\varepsilon_{j} ).
\end{eqnarray*}

We also have 
\begin{eqnarray*}
(\tilde{B}^{2}C^{1})(\tilde{j}, j ) 
        &  \geq & \sum_{\tilde{i}=1}^{I_{\tilde{n}}} 
        \tilde{s}^{-2} \tilde{m} \tilde{K}  \sigma_{2} 
        \circ \gamma_{\tilde{n}}(\varepsilon_{\tilde{i}}) 
         (C^{1}\varepsilon_{j})(\tilde{i})\\
           &  =  &  \tilde{s}^{-2} \tilde{m} \tilde{K}  \sigma_{2} 
        \circ \gamma_{\tilde{n}}(C^{1}\varepsilon_{j})
\end{eqnarray*}
and as above
\begin{eqnarray*}
(\tilde{B}^{2}C^{1})(\tilde{j}, j ) 
        &  \leq & \tilde{s}^{2} (\tilde{m} +1) \tilde{K}  \sigma_{2} 
        \circ \gamma_{\tilde{n}}(C^{1}\varepsilon_{j}) \\
          &  = & \tilde{s}^{3} \tilde{m} \tilde{K}  \sigma_{2} 
        \circ \gamma_{\tilde{n}}(C^{1}\varepsilon_{j}).
\end{eqnarray*}

We observe that  if $l \geq l_{0} > m$, then 
$l+1 = l(1 + l^{-1}) \leq l ( 1 + l_{0}^{-1}) \leq ls$ and 
$(1-l^{-1})^{-1} \leq s$.

Then, combining
the inequalities above yields the  first line of part 1.
Similar computations which we omit establish the 
 inequalities of the second line of part 1.

For the second part, we begin with the estimate in line (\ref{eq:92})
 above
and follow it by part 2 of Proposition \ref{proof:50}:
\begin{eqnarray*}
\label{eq:100}
(\tilde{B}^{1}C^{1})(\tilde{j}, j ) &  \leq & \tilde{s}^{3}\tilde{m}  \tilde{K}
     \sigma_{1} \circ \gamma_{\tilde{n}}(C^{1}\varepsilon_{j} ) \\
    &  \leq & \tilde{s}^{3}\tilde{m} \tilde{K}
     \sigma_{1} \circ \gamma_{\tilde{n}}(C^{1}\varepsilon_{j} ) 
      l^{-1} K^{-1} B^{1}(lv^{1} + v^{2})(j) \\
  & = &  \tilde{s}^{3}\tilde{m} \tilde{K}  l^{-1}  K^{-1} 
\sigma_{1} \circ \gamma_{\tilde{n}}(C^{1}\varepsilon_{j}) B^{1}(lv^{1} + v^{2})(j) ).
\end{eqnarray*}
We focus on the last term, summing over $j=1, \ldots, J$ and using
the inequality from part 2 of Proposition \ref{proof:60}
\begin{eqnarray*}
\sum_{j=1}^{J}
 \sigma_{1} \circ \gamma_{\tilde{n}}(C^{1}\varepsilon_{j} 
 B^{1}(lv^{1} + v^{2})(j) ) 
    & = &  \sigma_{1} \circ \gamma_{\tilde{n}}(C^{1}B^{1}(lv^{1} + v^{2})) \\
     &  \leq & s \sigma_{1} \circ \gamma_{\tilde{n}}
     (lu^{1}_{\tilde{n}} + u^{2}_{\tilde{n}}) \\
       &  =  &  s \sigma_{1}(lu^{1} + u^{2}) \\
       & = & sl.
\end{eqnarray*}
Therefore, we have
\begin{eqnarray*}
(\tilde{B}^{1}C^{1})(\tilde{j}, j ) &  \leq & 
\tilde{s}^{3}\tilde{m} \tilde{K}  l^{-1}  K^{-1} \sum_{j=1}^{J}
\sigma_{1} \circ 
\gamma_{\tilde{n}}(C^{1}\varepsilon_{j}) B^{1}(lv^{1} + v^{2})(j) ) \\
   &  \leq & s \tilde{s}^{3}\tilde{m}\tilde{K}K^{-1}.
\end{eqnarray*}
The other  inequality for $\tilde{B}^{1}C^{1}$
is obtained using part 1 and
part 1 of Proposition \ref{proof:50}. The two other inequalities for
 $\tilde{B}^{2}C^{2}$ are obtained in a similar way.
\end{proof}

We now turn to the proof of Theorem \ref{intro:40}.
 We will construct the following diagram.
  \[
\xymatrix{
 \Z^{I_{1} }  \ar[rr]^{A_{n_{1}-1} \cdots A_{1}} 
 \ar[rd]_{ B_{1}} &    &  \Z^{I_{n_{1}} }  
 \ar[rd]_{ B_{2}} \ar[rr]^{A_{n_{2}-1} \cdots A_{n_{1}}}
  &  &  \Z^{I_{n_{2}} }  \ar[r] 
 \ar[rd]_{ B_{3}} &    \\
 & \Z^{J_{1}} \oplus \Z^{J_{1}}  \ar[ru]_{C_{1}} 
 \ar[rr]^{B_{2}C_{1}} &  &  \Z^{J_{2}} \oplus \Z^{J_{2}} 
  \ar[ru]_{C_{2}} \ar[rr]^{B_{3}C_{2}}   &  & 
 \Z^{J_{3}} \oplus \Z^{J_{3}}  } 
\]
 The positive integers $I_{n}$ and the positive integer 
 matrices $A_{n}$ are the ones with which we began.
 
 The values of $J_{i}, n_{i}$ and matrices 
 $B_{i} = \left[ \begin{matrix} B_{i}^{1} 
 \\ B_{i}^{2} \end{matrix} \right]$
  and  
  $C_{i} = \left[ \begin{matrix} C_{i}^{1} 
  & C_{i}^{2} \end{matrix} \right]$
  are then obtained inductively. Specifically,
   $J_{i}, B_{i}$ are 
  obtained by 
  application of 
  Proposition \ref{proof:50}, while $n_{i}, C_{i}$ are 
  obtained by 
  application of 
  Proposition \ref{proof:60}.
  
  Let us proceed more carefully. 
  We begin by applying Proposition \ref{proof:50} using 
  $I=I_{1}, K_{0} = 4$
  and $m$ chosen so that 
  $s = 1 + m^{-1} < r_{1}^{1/3}$.  
   We also use $v^{1}=u_{1}^{1}, v^{2}=u_{1}^{2}$.
The result provides us with 
$l \geq l_{0}, K \geq K_{0}, J$ and matrices $B^{1}$ and $B^{2}$.
   These we define as $J_{1}=J, l_{1}=l, K_{1} = K, B^{1}_{1}= B^{1}$ 
   and $B^{2}_{1}=B^{2}$. 
We also define $u^{1,1}_{1}= B^{1}v^{1},  u^{1,2}_{1}= B^{2}v^{1}, 
 u^{2,1}_{1}= B^{1}v^{2},u^{2,2}_{1}= B^{2}v^{2}$.
 
 Among the conclusions, we note that the fourth property 
 of Proposition 
 \ref{proof:50} along with the choice  $s = 1 + m^{-1} < r_{1}^{1/3}$
   implies condition 4 of 
 Theorem \ref{intro:40} for $n=1$.
 
 Next, given this data, we apply Proposition \ref{proof:60}
 which gives us $\tilde{n} > 1$ and matrices $C^{1}, C^{2}$.
 We will define $n_{1}=\tilde{n}, C^{1}_{1} = C^{1}, C^{2}_{1} = C^{2}$.
 
 We now return to a second application of Proposition \ref{proof:50}.
 As we indicated in our notation early, we will use
 as input, 
 $\tilde{I}=I_{n_{1}}, \tilde{K}_{0}= 4K_{1},  \tilde{a}= a_{1},
  \tilde{m}$ and $\tilde{l}_{0}$.
 We use $v^{1} = u_{n_{1}}^{1}, v^{2} = u_{n_{1}}^{2}$.
 We will choose $\tilde{m} $ such
  that $\tilde{s} = 1 + \tilde{m}^{-1}$ satisfies
 $\tilde{s} < \min \{ r_{2}^{1/3}, r_{1}^{1/9} \}$.
We also choose $\tilde{l_{0}} > \tilde{m}$ and also 
 $\tilde{l_{0}} \geq l_{1}$.
The application gives 
 us $\tilde{l} \geq \tilde{l}_{0}, \tilde{K} \geq \tilde{K}_{0}, \tilde{J}$ 
 and matrices
 $\tilde{B}^{1}, \tilde{B}^{2}$. We define 
 $B^{1}_{2}= \tilde{B}^{1}, B^{2}_{2} = \tilde{B}^{2}$, $K_{2} = \tilde{K}, 
 l_{2} = \tilde{l}$ and $M_{1} = \tilde{m} \tilde{K} K_{1}^{-1}$.
 We also define $u^{1,1}_{2}= \tilde{B}^{1}v^{1},  u^{1,2}_{2}= 
 \tilde{B}^{2}v^{1}, 
 u^{2,1}_{2}= \tilde{B}^{1}v^{2},u^{2,2}_{2}= \tilde{B}^{2}v^{2}$. 
  Notice that the entries
 of the matrices $\tilde{B}^{1}, \tilde{B}^{2}$ are all bounded below
 by $a_{1}$ and since $C^{1}, C^{2}$ have positive integer entries,
 the same is true of the products $\tilde{B}^{1}C^{2}, 
 \tilde{B}^{2}C^{1}$.

 To verify the desired conditions, once again estimate 4 holds for
 $n=2$ from the third part of the conclusion of  Proposition \ref{proof:50}
 along with our choice  $\tilde{s} <  r_{2}^{1/3}$.
 
 The first inequality of part 5 has been noted above already. 
 The others follow from the first part of the 
 conclusion of Proposition \ref{proof:70}
 and the fact that 
 $s\tilde{s}^{5} < r_{1}^{1/3}(r_{1}^{1/9})^{5} = r_{1}^{8/9}< r_{1}$.
 Part 6 of the conclusion follows from the second part of the 
 conclusion of Proposition \ref{proof:70} and  the estimate above.
 
 The rest of the inductive system is 
 obtained by repetition of these arguments.
 We finally mention that in further applications 
 of Proposition \ref{proof:50}, we will choose $m $ greater than
 the last 
 $M_{n}$ which has been defined. Similarly, 
 in the applications, we will also use
 $K_{0} $ to be greater than the last $K_{n}$ which was defined. 
 These choices will ensure that the sequences $M_{n}$ and $K_{n}$
 are  increasing.
 
 This completes the proof of Theorem \ref{intro:40}.

\bibliographystyle{amsplain}

\begin{thebibliography}{99}

\bibitem{Bra:BD} O. Bratteli,  
 \textit{Inductive limits of finite dimensional $C^{*}$-algebras},
  Trans. Amer. Math. Soc. \textbf{171} (1972), 195-234.




\bibitem{CS:cohom} A. Clark and L. Sadun, 
\textit{Small cocycles, fine torus 
fibrations, and a $\Z^{2}$-subshift with neither},  Annales Henri
 Poincar\'{e} \textbf{18} (2017), 2301-2326.



\bibitem{DP:book} F. Durand and D. Perrin, \textit{Dimension
groups and dynamical systems}, Cambridge Studies in Advanced Mathematics 
\textbf{196}, Cambridge University Press, Cambridge, 2022.

\bibitem{Eff:CBMS}
E.G. Effros, \textit{Dimensions and $C^{*}$-algebras}, CBMS Regional 
Conference Series in Mathematics \textbf{46}, 1981.

\bibitem{Ell:AF}
G. A. Elliott, 
\textit{On the classification of inductive limits
 of sequences of semisimple finite
dimensional algebras}, 
 J. Algebra \textbf{38}
(1976), 29-44.




 \bibitem{FM:2states}
T. Fack and O. Mar\'{e}chal, \textit{Sur la 
classification des sym\'{e}tries des C *-alg\`{e}bres UHF},
Can. Math. J. \textbf{31} (1979), 496-523.



\bibitem{FS:fusion} N. Frank and L. Sadun, \textit{Fusion: a general
 framework for hierarchical tilings of $\R^{d}$}, 
 Geometriae Dedicata \textbf{171} (2014), 149-186.



\bibitem{GMPS:orbd} T. Giordano, H. Matui, I.F. Putnam
 and C.F. Skau, \textit{Orbit equivalence for Cantor 
 minimal $Z^{d}$-systems}, Invent. Math., \textbf{179} (2010), 119-158. 

\bibitem{GMPS:orb2} T. Giordano, H. Matui, I.F. Putnam
 and C.F. Skau, \textit{Orbit equivalence for Cantor minimal 
 $\Z^{2}$-actions}, Journal A.M.S., \textbf{21} (2008), 863-892.

\bibitem{GPS:orb} T. Giordano, I.F. Putnam and C.F. Skau,
\textit{Topological orbit equivalence and C*-crossed products}, 
 J. Reine Angew. Math. \textbf{469} (1995), 51-111.
 
\bibitem{GPS:odom}  T. Giordano, I.F. Putnam and C.F. Skau, 
\textit{$\Z^{d}$-odometers and cohomology},
 Groups, Geometry, and Dynamics \textbf{13} (2019), no. 3,  909-938.


\bibitem{Good:book} K. Goodearl, \textit{Partially ordered abelian groups
with interpolation}, Mathematical Surveys and Monographs 
\textbf{20}, American Mathematical Society, Providence, 1986.



\bibitem{HPS:model} R.H. Herman, I.F. Putnam and C.F. Skau, 
\textit{Ordered Bratteli Diagrams, Dimension groups and
topological dynamics},  Int. J. Math.\textbf{ 03} (1992), 827-864. 




\bibitem{Kri:paper} W. Krieger, \textit{On a dimension for a class
of homeomorphism groups}, Math. Ann. \textbf{252} (1980), 87-95.


\bibitem{Put:CMS} I.F. Putnam,  \textit{Cantor Minimal Systems},
University Lecture Series \textbf{70},
American Mathematical Society, Providence, RI, 2018.



\bibitem{Ren:LNM} J. Renault, \textit{A Groupoid Approach to \,
$C^{*}$-algebras}, Lecture Notes in
Mathematics 793, Springer, Berlin, 1980.
 

\bibitem{SV:Uinfin} S. Stratila and D. Voiculescu, 
\textit{Representations of AF-Algebras and of the Group $U(\infty)$},
Lecture Notes in Mathematics \textbf{486}, Springer, Berlin-Heidelberg, 1975. 



\bibitem{Thom:class} S. Thomas, \textit{The classification 
problem for  torsion-free abelian groups of finite rank}
J. American Mathematical Society \textbf{16} (2002), 233-258.


\bibitem{Ver:adic} A. M. Vershik,
\textit{Uniform algebraic approximations of shift and multiplication operators},
Dokl. Akad. Nauk SSSR 259  (1981), No.3, 526-529,
English translation: Sov. Math. Dokl.(1981) 24, 97-100. 

\bibitem{VerL:Zd} A. M. Vershik and A.A. Lodkin, 
\textit{Approximation for actions of amenable groups and transversal 
automorphism}, Lecture Notes in Mathematics \textbf{1132}, Springer, Berlin, 1985, 331-346. 

\bibitem{Ver:Mark} A. M. Vershik, \textit{A Theorem on the Markov periodical 
approximation in ergodic theory}, J. Soviet Math. 
\textbf{28} (1985), 667-674.


\end{thebibliography}

\end{document}